\documentclass{article}

\usepackage{amssymb,color,bbm,amsmath,amsthm,geometry,mathtools,cleveref}
\allowdisplaybreaks

\newtheorem{theorem}{Theorem}[section]
\newtheorem{lemma}[theorem]{Lemma}
\newtheorem{proposition}[theorem]{Proposition}

\newtheorem{definition}[theorem]{Definition}
\newtheorem{remark}[theorem]{Remark}

\newcommand{\dx}{{\mathrm d}x}
\newcommand{\dt}{{\mathrm d}t}
\newcommand{\dxdt}{\dx\dt}
\newcommand{\Om}{\Omega}

\newcommand{\e}{\varepsilon}

\newcommand{\diver}{\mathrm{div}}

\newcommand{\grad}{\nabla}

\newcommand{\bfu}{{\bf u}}
\newcommand{\bfU}{{\bf U}}
\newcommand{\bff}{{\bf f}}

\newcommand{\bfW}{{\bf W}}

\newcommand{\bfv}{{\bf v}}
\newcommand{\bfxi}{{\boldsymbol{\xi}}}
\newcommand{\Pt}{\partial_{t}}
\newcommand{\weak}{\rightharpoonup}

\newcommand{\cost}{\mathcal{J}}
\newcommand{\adcont}{\mathcal{U}_{\mathrm{ad}}}

\newcommand{\adall}{\mathcal{A}_{\mathrm{ad}}}

\begin{document}



\title{Necessary conditions for distributed optimal control of linearized compressible Navier-Stokes equations with state constraints}

\author{Stefan Doboszczak, Manil T. Mohan, and Sivaguru S. Sritharan
}
\date{}
\maketitle

\begin{abstract}
A Pontryagin maximum principle for an optimal control problem in three dimensional linearized compressible viscous flows is established using the Ekeland variational principle. The controls are distributed over a bounded domain, while the state variables are subject to a set of constraints and governed by the linearized compressible Navier-Stokes equations. The maximum principle is of integral-type and obtained for minimizers of a tracking-type integral cost functional. 
\end{abstract}



\section{Introduction}
\subsection{Governing equations}
Let $Q_T=(0,T)\times\Om$ denote a space-time domain, where 
$T>0$ is fixed and $\Om\subset\mathbb{R}^3$ is bounded. The state of a compressible and viscous fluid in $Q_T$ may be modeled by the \emph{compressible Navier-Stokes} equations,
	\begin{subequations}\label{NS_01}\begin{align}
		\Pt\rho + \diver(\rho\bfu)&=0,\label{l02}\\
		\rho(\Pt\bfu+\bfu\cdot\grad\bfu)+\grad p &= \diver\,\mathbb{S}+\rho\bff,\label{l01}
	\end{align}\end{subequations}
consisting respectively of the conservation of fluid mass and fluid momentum,
where $\rho(t,x)\in\mathbb{R}^+$ denotes the fluid density and $\bfu(t,x)\in\mathbb{R}^3$ denotes the fluid velocity. 
The body force (acceleration) $\bff(t,x)\in\mathbb{R}^3$ is fixed, denoting for instance a gravitational force. The system \eqref{NS_01} is supplemented with constitutive laws for the pressure $p(\rho)$ and viscous stress tensor 
$\mathbb{S}(\grad\bfu)$, as well as initial and boundary conditions.

While the nonlinear system \eqref{NS_01} describes the general evolution of barotropic compressible fluids, the focus of this paper is on a control problem for the \emph{linearized compressible Navier-Stokes} equations. Employing the standard linearization about a fixed state $(\widetilde{\rho},\widetilde{\bfu})$ satisfying \eqref{NS_01}, we formulate the (controlled) linearized equations as follows,
	\begin{subequations}\label{NSlin}\begin{align}
		\Pt\rho + \diver(\rho\widetilde{\bfu})+\diver(\widetilde{\rho}\bfu)&=0, \label{l1}\\
		\Pt\bfu+ L(\rho,\bfu)
		&= (\widetilde{\rho})^{-1}[\diver\,\mathbb{S}(\grad\bfu)+\rho\bff+\bfU] \label{l2},
	\end{align}\end{subequations}
where a \emph{distributed control} $\bfU$ has been introduced in the linearized momentum equation \eqref{l2}. The linear operator $L$ is given by
	\begin{equation*}
		L(\rho,\bfu)
		=\bfu\cdot\grad\widetilde{\bfu}+\widetilde{\bfu}\cdot\grad\bfu+(\widetilde{\rho})^{-1}[\grad(\rho p'(\widetilde{\rho}))+\rho(\Pt\widetilde{\bfu}+\widetilde{\bfu}\cdot\grad\widetilde{\bfu})].
	\end{equation*}
Here $p'(\cdot)$ denotes differentiation with respect to the argument. The solution $(\rho,\bfu)$ of the linearized system without control may be understood as representing the dynamics of a small perturbation of the state $(\widetilde{\rho},\widetilde{\bfu})$.

The linearized system \eqref{NSlin} is supplemented with a general barotropic pressure law,
	\begin{equation*}
		p = p(\widetilde{\rho}).
	\end{equation*} 
Note that in \eqref{l2} the pressure appears only as a coefficient, and is independent of $\rho$. 

The viscous stress tensor $\mathbb{S}$ is given by the Newtonian law
	\begin{equation}\label{stress}
		\mathbb{S}(\grad\bfu) 
		= \mu\left(\grad\bfu+(\grad\bfu)^T-\frac{2}{3}\diver\,\bfu\mathbb{I}\right)+\eta\diver\,\bfu\mathbb{I}, 
	\end{equation}
where the constant $\mu>0$ is the shear viscosity, and the constant $\eta\ge 0$ is the bulk viscosity. Defining
$\displaystyle{\lambda := \eta-(2/3)\mu}$, the diffusive term in \eqref{l2} may be written
	\begin{equation*}
		\diver\,\mathbb{S}(\grad\bfu)=\mu\Delta\bfu + (\mu+\lambda)\grad\diver\,\bfu.
	\end{equation*}
We also suppose that $4\mu+3\lambda >0$ (cf. \cref{lame}). 

Denoting the space-time boundary $\Gamma_T = (0,T)\times\partial\Om$, we furthermore suppose the fluid satisfies the no-slip boundary condition,
	\begin{equation}\label{boundary}
		\bfu(t,x)= 0\quad\hbox{on}\,\,\Gamma_T.
	\end{equation}
Finally, a set of initial conditions are imposed,
	\begin{equation*}
		\bfu(0,x)=\bfu_0(x),\,\,\rho(0,x)=\rho_0(x).
	\end{equation*} 
\subsection{Notation}
The space-time domain $Q_T=(0,T)\times\Om$ is fixed, where $T>0$, and $\Om\subset\mathbb{R}^3$ is an open, bounded subset. 
The boundary of $Q_T$ is denoted $\Gamma_T = (0,T)\times\partial\Om$. 
Since $T$ and $\Om$ are fixed, we use the abbreviated notation $L^p_t(L^q_x)= L^p(0,T;L^q(\Om;\mathbb{R}^N))$ 
and $L^p_t(W^{k,p}_{x})= L^p(0,T;W^{k,p}(\Om;\mathbb{R}^N))$ to denote the standard Bochner spaces endowed with their
strong topologies (the range $\mathbb{R}^N$ will be clear from context). We also write $H^k_x= W^{k,2}_x$ when $p=2$, and for $H^1_x$ functions vanishing on the boundary in the sense of trace we write $H^1_{0,x}$. By $C([0,T];X)$ we denote functions continuous on $[0,T]$ with respect to the strong topology on a Banach space $X$.

For a vector quantity $\bfu\in\mathbb{R}^3$, $\grad\bfu$ is the Jacobian matrix, while the Hessian $\grad^2\bfu$ is a third-order tensor.
The tensor product ${\bf{a}}\otimes{\bf{b}}$ of two vectors ${\bf{a}}$ and ${\bf{b}}$ is a second-order tensor defined componentwise 
by $a_i b_j$ ($i,j=1,2,3$). For two second-order tensors ${\bf{A}}$ and ${\bf{B}}$, we denote their Frobenius inner product by 
${\bf{A}}\!:\!{\bf{B}}=\sum_{i,j=1}^3 A_{ij}B_{ij}$. By $C$ we denote an arbitrary constant which may change values. Occasionally the dependence of $C$ on other constants will be clearly specified. The $L^2(\Om)$ inner product will be denoted $\langle\cdot,\cdot\rangle$. Finally, $a\lesssim b$ means that $a\le Cb$ for some positive constant $C$. Other notation will be introduced as necessary.

\subsection{Strong solutions of the governing equations}
The following theorem provides the existence of strong solutions for the linearized system \eqref{NSlin}.
	\begin{theorem}\label{mainthm}
		Fix $q\ge 3$. Let $Q_T$ be fixed with $\partial\Om\in C^\infty$. Let $(\widetilde{\rho},\widetilde{\bfu})$  
		be a smooth solution up to the boundary of the nonlinear system \eqref{NS_01} 
		emanating from initial data $(\widetilde{\rho}_0, \widetilde{\bfu}_0)$, 
		where $\mathbb{S}(\grad\widetilde{\bfu})$ is given by \eqref{stress}, 
		$p\in C^2(0,\infty)$, $\widetilde{\bfu}$ satisfies the no-slip boundary condition \eqref{boundary}, 
		and $\bff\in L^2_t(L^q_x)$. 
		Suppose furthermore that for all $(t,x)\in Q_T$, there exist constants $m$ and $M$  
		such that 
			\[0<m\le \widetilde{\rho}(t,x)\le M<\infty.\]
		Let $\rho_0\in H^1_x$ and $\bfu_0\in H^1_{0,x}$, and suppose $\bfU\in L^2_t(L^2_x)$. Then there exists a
		unique strong solution $(\rho,\bfu)$ 
		of the linearized system \eqref{NSlin}, and supplemented with the no-slip 
		condition \eqref{boundary}, such that
			\begin{equation}\label{mainest}\begin{aligned}
				&\rho\in L^\infty_t(H^1_x),\quad \Pt\rho\in L^\infty_t(L^2_x),\\
				&\bfu\in L^\infty_t(H^1_{0,x})\cap L^2_t(H^2_x),\quad \Pt\bfu \in L^2_t(L^2_x).
			\end{aligned}\end{equation}
	\end{theorem}
The proof of \cref{mainthm} is given in the Appendix. See \cite{gl87, neu98} for alternate existence results with different assumptions on the data.
\begin{remark}
Unlike the nonlinear system \eqref{NS_01}, there is no reason that the density $\rho$ of the linearized system should be nonnegative. Since $\rho$ typically represents a perturbation about $\widetilde{\rho}$, this is not unexpected.
\end{remark}
\begin{remark}
The smoothness requirement of $(\widetilde{\rho},\widetilde{\bfu})$ is chosen for simplicity and can be considerably weakened. For instance, it is sufficient to consider $(\widetilde{\rho},\widetilde{\bfu})$ having the regularity specified in Propositions \ref{prop1} and \ref{prop2}. For the same reason the pressure regularity may be reduced to $p\in W^{2,3}_{loc}(0,\infty)$. Note the existence of such regular solutions to the nonlinear system for general forcings and large initial data is unknown.
\end{remark}

\subsection{Optimal control problem}
We now describe the optimal control problem addressed in this paper. The objective is to determine control-state triples $(\rho,\bfu,\bfU)$ that minimize the following tracking-type \emph{cost functional},
	\begin{equation}\label{cost}
		\cost(\rho,\bfu,\bfU) = \frac{1}{2}\int_0^T\!\left(\|\bfu-\bfu_d\|^2_{L^2_x}+\|\rho-\rho_d\|^2_{L^2_x}
		+\|\bfU\|^2_{L^2_x}\right)\dt,
	\end{equation}
where desired target states are denoted by $\rho_d\in  L^2_t(L^2_x)$ and $\bfu_d\in L^2_t(L^2_x)$. 

In addition, we specify a \emph{state constraint} of the form 
	\begin{equation}\label{state}
	F(\rho,\bfu)\in W,
	\end{equation}
where the mapping $F: L^2_t(L^2_x)\times L^2_t(H^1_{0,x})\rightarrow X$ is assumed to be continuously Fr{\'e}chet differentiable, $X$ is a Banach space with strictly convex dual $X^*$, and $W\subset X$ is a nonempty closed and convex subset. 	
	
We denote by $\overline{B}_R(0)$ the closed ball of $L^2_x$,
\[\overline{B}_R(0) := \{u\in L^2(\Om;\mathbb{R}^3)\,:\, \|u\|_{L^2_x}\le R\},\] 
where $R>0$ is fixed. The set of controls is taken as follows,
	\begin{equation*}
	\adcont := L^2(0,T;\overline{B}_R(0)).
	\end{equation*}
	
\begin{definition}\label{defall}
The admissible class $\adall$ is defined as the set of state-control triples $(\rho,\bfu,\bfU)$, with $\bfU\in\adcont$, solving the linearized compressible Navier-Stokes system \eqref{NSlin} in the sense of \cref{mainthm}, and satisfying the state constraint \eqref{state}.
\end{definition}
The first order necessary conditions in the form of a Pontryagin maximum principle will be obtained for the following optimal control problem:
	\begin{equation}\label{ptag}\tag{P1}\begin{aligned}
		&Minimize\,\,\,\cost(\rho,\bfu,\bfU)\,\,{over\,\, all}\,\,(\rho,\bfu,\bfU)\in\adall
	\end{aligned}\end{equation}
where it is understood that the admissible class includes the PDE contraint, \cref{NSlin} and the state constraint $F(\rho,\bfu)\in W$.
\begin{definition}
A solution of Problem \eqref{ptag} (provided it exists) is called an optimal solution and denoted $(\rho^*,\bfu^*,\bfU^*)$. The control $\bfU^*$ is called an optimal control.
\end{definition}
\begin{remark}
Since the state $(\rho,\bfu)$ is determined uniquely by the control $\bfU$ through the governing PDE, i.e. $\rho=\rho[\bfU]$ and $\bfu = \bfu[\bfU]$, the cost $\cost$ may be considered as a function of the control only.
\end{remark}
Optimal control theory of incompressible viscous flow has an extensive literature, see for instance Fursikov \cite{f00}, Gunzburger \cite{g03}, Lions \cite{lions71}, Sritharan \cite{s98}, and the references therein. In this paper we extend the techniques developed in \cite{fs94, f99, ww03} for incompressible viscous flow to the linearized compressible case.

Rigorous studies of optimal control problems for compressible Navier-Stokes are less numerous than their incompressible counterpart. Existing mathematical results tend to focus on problems with a number of simplifications including working with the linearized rather than nonlinear system, or possibly employing some non-physical assumptions in order to make problems tractable.

An optimal control problem for linearized compressible Navier-Stokes flows is considered in \cite{cr13}. The controls act on the boundary of a two-dimensional inlet-outlet domain and the existence and necessary conditions are established. Existence of optimal controls for the nonlinear compressible system are established in \cite{dms17}, where the authors use a stronger cost functional and weak-strong uniqueness result to ensure the uniqueness of the state variables. In \cite{amasova09} an optimal control problem is formulated for a Navier-Stokes-Fourier system in one dimension for an inlet-outlet domain in Lagrangian coordinates. For results on null and approximate controllability, and stabilizability, see \cite{crr12,fgip04,mrr17} and the references therein.

Computational implementation of optimality conditions have been addressed in \cite{cghuu}, motivated by problems in aeroacoustics, and \cite{jpm98} for problems in optimal design.

An advantage to using the linearized version of the compressible flow equations is that more practical cost functionals may be chosen, while at the same time ensuring well-posedness of the governing equations in a strong enough sense. Strong solutions are known to exist for the nonlinear system
\eqref{NS_01} (cf. \cite{ck06, mn80, valli83}), though the required regularity of the forcing tends to be much stronger than is often required in control-theoretic problems. Since discontinuous controls are relevant in practice, we seek to avoid such regularity requirements. 

The lack of suitably strong solutions of the nonlinear system \eqref{NS_01} introduces other difficulties relevant to solvability of the adjoint and linearized equations, used in obtaining the necessary conditions. A number of authors 
(cf. \cite{bella15, cr13}) have observed that often more regularity is needed for the base state of the linearization. Collis et al. \cite{cghuu} similarly observe that ``...the solution of the linearized state equation is less regular than the state . . . about which the linearization is
done.'' In order to avoid such issues, we ensure that the base state $(\widetilde{\rho},\widetilde{\bfu})$ is sufficiently regular.

In order to derive the necessary conditions for Problem \eqref{ptag}, we make use of Ekeland's variational principle, due to I. Ekeland \cite{eke74}. This tool provides the existence of minimizers for rather general cost functionals, and is useful when precise extrema are not required, or do not exist. In the context of incompressible fluids, the variational principle has been used by many authors to obtain the necessary conditions, cf. \cite{fs94, fs98,ww03}. The reference \cite{fs98} in particular deals with incompressible flow with state constraints.

We anticipate relevance of this paper to applications including aerodynamic design \cite{bb97, jpm98}, and noise control \cite{cgh2001,oss16}, as well as to the general theory of optimal control of compressible fluids.

\section{{The necessary conditions and main theorem}}
In this section we formally derive the Pontryagin maximum principle for Problem \eqref{ptag}. In order for the state evolution to be described explicitly, we divide the momentum equation by $\widetilde{\rho}$, which we recall is assumed to be strictly positive.

Begin by defining 
	\begin{equation*}\begin{aligned}
		-\mathcal{N}_1(\rho,\bfu) &=  \diver(\rho\widetilde{\bfu})+\diver(\widetilde{\rho}\bfu),\\
		-\mathcal{N}_2(\rho,\bfu,\bfU) &=\bfu\cdot\grad\widetilde{\bfu}+\widetilde{\bfu}\cdot\grad\bfu
		+(\widetilde{\rho})^{-1}\grad(\rho p'(\widetilde{\rho}))\\ 
		&\quad +(\widetilde{\rho})^{-1}
		\left[\rho(\Pt\widetilde{\bfu}+\widetilde{\bfu}\cdot\grad\widetilde{\bfu})-\diver\,\mathbb{S}(\grad\bfu)-\rho\bff-\bfU\right].
	\end{aligned}\end{equation*}
The linearized system \eqref{NSlin} may be concisely written
	\begin{equation*}\left\{\begin{aligned}
		\Pt\rho &= \mathcal{N}_1(\rho,\bfu)\\
		\Pt\bfu &= \mathcal{N}_2(\rho,\bfu,\bfU).
	\end{aligned}\right.\end{equation*}
We define the \emph{augmented cost functional}, $\widehat{\mathcal{J}},$ as 
	\begin{equation*}\begin{aligned}
		\widehat{\mathcal{J}}(\rho,\bfu,\sigma,\bfxi,\bfU) := \lambda\cost(\rho,\bfu,\bfU)
		&+\int_0^T\left\langle\sigma,\Pt\rho-\mathcal{N}_1(\rho,\bfu)\right\rangle\dt\\
		&+\int_0^T\left\langle\bfxi,\Pt\bfu-\mathcal{N}_2(\rho,\bfu,\bfU)\right\rangle\dt
		+ \zeta\cdot d_W(F(\rho,\bfu)),
	\end{aligned}\end{equation*}
where $\sigma$ and $\bfxi$ denote adjoint variables to $\rho$ and $\bfu$ respectively, $\mathrm{d}_W$ denotes distance to the set $W$ in the norm of $X$, $\lambda$ is a multiplier for the cost functional, and $\zeta$ is a multiplier for the state constraint.

Let $(\rho^*,\bfu^*,\bfU^*)$ be an optimal triple. Formally, the adjoint equations are derived by differentiating $\widehat{\cost}$ in the G{\^a}teaux sense with respect to the state-control variables and taking adjoints,
	\begin{equation}\label{adjeq}\begin{aligned}
		-\Pt\sigma + \lambda\partial_\rho{\mathcal{J}} - [\partial_\rho\mathcal{N}_1]^*\sigma - [\partial_\rho\mathcal{N}_2]^*\boldsymbol{\xi} &= -\zeta[F_\rho(\rho^*,\bfu^*)]^*\eta,\\
		-\Pt\bfxi + \lambda\partial_\bfu{\mathcal{J}} - [\partial_\bfu\mathcal{N}_1]^*\sigma - [\partial_\bfu\mathcal{N}_2]^*\boldsymbol{\xi} &= -\zeta[F_\bfu(\rho^*,\bfu^*)]^*\eta,\\
		\lambda\partial_\bfU{\mathcal{J}} - [\partial_\bfU\mathcal{N}_2]^*\boldsymbol{\xi} &= 0,
	\end{aligned}\end{equation}
where $\eta\in X^*$ is in the subdifferential of $\mathrm{d}_W$ at the point $F(\rho^*,\bfu^*)$, i.e. $\eta\in \partial\mathrm{d}_W(F(\rho^*,\bfu^*))$.
\begin{remark}
Differentiating $\widehat{\cost}$ with respect to $\sigma$ and $\boldsymbol{\xi}$ recovers the original system.
\end{remark}
From \eqref{adjeq} we obtain that $(\sigma,\bfxi)$ should satisfy the following adjoint system on $Q_T$,
	\begin{equation}\label{adj0}\begin{aligned}
		&-\Pt\sigma-\widetilde{\bfu}\cdot\grad\sigma
		=(\widetilde{\rho})^{-1}\bfxi\cdot[-(\Pt\widetilde{\bfu}+\widetilde{\bfu}\cdot\grad\widetilde{\bfu})+\bff] +p'(\widetilde{\rho})\diver\left((\widetilde{\rho})^{-1}\bfxi\right)\\
		&\hspace{80pt}+\lambda(\rho_d-\rho^*)-\zeta[F_\rho(\rho^*,\bfu^*)]^*\eta,\\
		&-\Pt\bfxi-\diver(\bfxi\otimes\widetilde{\bfu})+\bfxi\cdot(\grad\widetilde{\bfu})^T =\widetilde{\rho}\grad\sigma+\diver\mathbb{S}(\grad((\widetilde{\rho})^{-1}\bfxi))+\lambda(\bfu_d-\bfu^*)-\zeta[F_\bfu(\rho^*,\bfu^*)]^*\eta,\\
		&\bfxi\big|_\Gamma = 0,\\
		&\sigma(T,\cdot)=0,\,\,\bfxi(T,\cdot)=0,
	\end{aligned}\end{equation}
and in addition the Pontryagin principle for optimal control holds,
	\begin{equation*}
		\frac{1}{2}\lambda\|\bfU^*(\tau)\|^2_{L^2_x}-\left\langle(\widetilde{\rho})^{-1}(\tau)\bfxi(\tau),\bfU^*(\tau)\right\rangle \le
		\frac{1}{2}\lambda\|{\bf{W}}\|^2_{L^2_x}-\left\langle(\widetilde{\rho})^{-1}(\tau)\bfxi(\tau),{\bf{W}}\right\rangle,
	\end{equation*}
to be satisfied for all ${\bf{W}} \in \overline{B}_R(0)$. Equivalently, this may be written in the Hamiltonian formulation
	\begin{equation*}
		\mathcal{H}(\rho^*(\tau),\bfu^*(\tau),\sigma(\tau),\bfxi(\tau),\bfU^*(\tau),\lambda) 
		= \min_{{\bf{W}}\in L^2_x}\mathcal{H}(\rho^*(\tau),\bfu^*(\tau),\sigma(\tau),\bfxi(\tau),{\bf{W}},\lambda),
	\end{equation*}
where we define the \emph{Hamiltonian}
	\begin{equation*}
		\mathcal{H}(\rho,\bfu,\sigma,\bfxi,\bfU,\lambda) := \lambda\mathcal{L}(\rho,\bfu,\bfU)+\left\langle\sigma, \mathcal{N}_1(\rho,\bfu)\right\rangle
		+\left\langle\bfxi,\mathcal{N}_2(\rho,\bfu,\bfU)\right\rangle,
	\end{equation*}
and corresponding \emph{Lagrangian}
	\[\mathcal{L}(\rho,\bfu,\bfU):=\frac{1}{2}(\|\bfu-\bfu_d\|^2_{L^2_x}+\|\rho-\rho_d\|^2_{L^2_x}+\|\bfU\|^2_{L^2_x}).\]
The following theorem states the main result of this paper regarding the necessary conditions for Problem \eqref{ptag}.
\begin{theorem}\label{mthm}
	Let $(\rho^*,\bfu^*,\bfU^*)$ be an optimal solution of Problem \eqref{ptag}. Then there exists $\lambda\in\mathbb{R}$, $\eta\in X^*$, and a weak solution 
	$(\sigma,\bfxi)$ of the adjoint system \eqref{adj0} in the sense of distributions, and furthermore for all ${\bf{W}}\in \overline{B}_R(0)$,
	\begin{equation}\label{t1}
	\frac{1}{2}\lambda\|\bfU^*\|^2_{L^2_t L^2_x}-\left\langle(\widetilde{\rho})^{-1}\bfxi,\bfU^*\right\rangle_{L^2_t L^2_x}
	\le \frac{1}{2}\lambda\|{\bf{W}}\|^2_{L^2_t L^2_x}-\left\langle(\widetilde{\rho})^{-1}\bfxi,{\bf{W}}\right\rangle_{L^2_t L^2_x},
	\end{equation}	
and \begin{equation}\label{cone1}
		\langle \eta,w-F(\rho^*,\bfu^*) \rangle \le 0\quad\forall w\in W.
	\end{equation}
where $\eta\in \partial\mathrm{d}_W(F(\rho^*,\bfu^*))$, $\partial d_W$ denoting the convex subdifferential of the distance function to the set $W$.
\end{theorem} 
\begin{definition}
Let $W\subset X$ be a convex subset of a Banach space $X$. The normal cone to $W$ at $\bar{w}$, denoted $N_W(\bar{w})$, is defined as 
	\[N_W(\bar{w}) = \{\eta \in X^* : \langle \eta, w- \bar{w} \rangle_{X^*;X}\le 0,\,\,\forall w\in W\}.\] 
\end{definition}
The inequality \eqref{cone1} therefore has the equivalent characterization $\eta\in N_W(F(\rho^*,\bfu^*))$.

\section{Mathematical preliminaries}
In this section we recall some relevant results to be used in the proof of the Pontryagin maximum principle.
\begin{lemma}
Let $X$ be a Banach space. The distance function $d_W$ to a nonempty subset $W\subset X$, defined as 
	\[d_W(x) = \inf_{w\in W}\|x-w\|_X,\]
is Lipschitz continuous with Lipschitz constant 1.
\end{lemma}
The result of this lemma is standard and also has a straightforward extension to metric spaces.
	\begin{theorem}[Ekeland variational principle]\label{th1}
		Let $(X,d)$ be a complete metric space, and $F: X\rightarrow \mathbb{R}\cup\{+\infty\}$ a lower semi-continuous functional, not identically $+\infty$, and bounded from below. Then for every point $u\in X$ such that
		\begin{equation}\label{emin}
			\inf_X F\le F(u)\le \inf_X F+\e
		\end{equation}
	and every $\lambda>0$, there exists some point $v\in X$ such that
		\begin{equation*}\begin{aligned}
			&F(v)\le F(u),\\
			&d(u,v)\le \lambda,\\
			&\forall w\neq v,\quad F(w)>F(v)-(\e/\lambda)d(v,w).
		\end{aligned}\end{equation*}
	\end{theorem}
\Cref{th1} is due to Ekeland \cite{eke74} 
(see \cite{eke79} for an alternative proof credited to M. Crandall).
	\begin{remark}
		A point $u\in X$ such that \eqref{emin} holds is called an $\e$-minimizer of $F$. 
	\end{remark}
Recall as the set of controls we choose
	\begin{equation}\label{controls}
	\adcont := L^2(0,T;\overline{B}_R(0)).
	\end{equation}
The set $\adcont$ is equipped with a suitable metric $d_E$, the Ekeland metric,
	\begin{equation*}
		d_E(u,v)= \mathrm{meas}(\{t\in [0,T]: u(t)\neq v(t)\}),
	\end{equation*}
which is defined for all $u,v\in\adcont$, and where the measure is Lebesgue measure on $\mathbb{R}$.
	\begin{lemma}\label{comlem}
		$(\adcont,d_E)$ is a complete metric space.
	\end{lemma}
The proof is classical and may be found in Ekeland \cite{eke74} (see also Lemma 3.15 \cite{fs94}). Note $\overline{B}_R(0)$ is closed, convex, and bounded. In general, \cref{comlem} is not valid if we allow unbounded control sets, i.e. $\adcont = L^2(0,T;U_{\mathrm{ad}})$ with an unbounded subset $U_{\mathrm{ad}}\subset L^2_x$ (see \cite{fs94}, pg. 227).

Furthermore, the following inequality holds,
	\begin{equation}\label{ineqc}
		\|\bfU_n-\bfU\|_{L^2_t(L^2_x)}\le 2R\hspace{1.25pt}[d_E(\bfU_n,\bfU)]^{1/2},
	\end{equation}
and so $\bfU_n\rightarrow \bfU$ strongly in $L^2_t(L^2_x)$ whenever $d_E(\bfU_n,\bfU)\rightarrow 0$.

The following variations are known as spike (needle) variations.
	\begin{definition}\label{def}
	 Fix $\tau\in(0,T)$. Let ${\bf{W}}\in \overline{B}_R(0)$ be arbitrary, and 
		let $h$ be chosen such that $0<h<\tau$. The spike variation $\bfU_{\tau,h,\bfW}$ of a control $\bfU\in\adcont$ is defined by 
			\begin{equation}\label{spike}
				\bfU_{\tau,h,\bfW}(t)=\left\{\begin{aligned}
				&\bfW,\hspace{18pt}\mathrm{if}\,\,t\in[0,T]\cap(\tau-h,\tau)\\
				&\bfU(t),\quad\mathrm{else}.
				\end{aligned}\right.\end{equation}
			For brevity, we denote the spike variation as $\bfU^h$, in which case $\tau$ and $\bfW$ are assumed fixed.
	\end{definition}
It is clear that $\bfU^h\in\mathcal{U}_{\mathrm{ad}}$ and $d_E(\bfU^h,\bfU)=h$.

The following lemma ensures the continuous dependence of solutions of \eqref{NSlin} on the controls $\bfU$.
	\begin{lemma}\label{nearby}
		Let $\bfU_n,\bfU\in \adcont$, and suppose $d_E(\bfU_n,\bfU)\rightarrow 0$ as $n\rightarrow\infty$. Let $(\rho_n,\bfu_n)$ be the solution of
		\eqref{NSlin} corresponding to the control $\bfU_n$, and $(\rho,\bfu)$ the solution corresponding to the control $\bfU$, 
		both emanating from the same initial data $(\rho_0,\bfu_0)$. Then 
		$(\rho_n,\bfu_n)\rightarrow (\rho,\bfu)$ strongly in the topology determined by the estimates \eqref{mainest} of 
		\cref{mainthm}. In particular, it holds at least that
			\begin{equation}\label{xx0}\begin{aligned}
				&\rho_n\rightarrow \rho\,\,\mathrm{strongly\,in}\,\, L^2_t(L^2_x),
				&\bfu_n\rightarrow \bfu\,\,\mathrm{strongly\,in}\,\, L^2_t(H^1_{0,x}).
			\end{aligned}\end{equation}
	\end{lemma}
	\begin{proof}
		Define $\tau_n=\rho_n-\rho$ and $\boldsymbol{\omega}_n=\bfu_n-\bfu$. Since the system \eqref{NSlin} is linear, it follows that $\tau_n$ and $\boldsymbol{\omega}_n$ satisfy the system
		\begin{equation}\label{lindiff}\begin{aligned}
			&\Pt\tau_n + \diver(\tau_n\widetilde{\bfu})+\diver(\widetilde{\rho}\boldsymbol{\omega}_n)=0,\quad a.e.\,\,\mathrm{in}\,\,Q_T, \\
			&\widetilde{\rho}(\Pt\boldsymbol{\omega}_n+\boldsymbol{\omega}_n\cdot\grad\widetilde{\bfu}
			+\widetilde{\bfu}\cdot\grad\boldsymbol{\omega}_n)
		+\grad(\tau_n p'(\widetilde{\rho}))\\ 
		&\qquad= -\tau_n(\Pt\widetilde{\bfu}+\widetilde{\bfu}\cdot\grad\widetilde{\bfu})
		+\diver\,\mathbb{S}(\grad\boldsymbol{\omega}_n)+\tau_n\bff+\bfU_n-\bfU,\quad a.e.\,\,\mathrm{in}\,\,Q_T\\
			&\boldsymbol{\omega}_n\big|_{\Gamma_T}=0,\\
			&\tau_n(0,\cdot)=0,\,\,\boldsymbol{\omega}_n(0,\cdot)=0.
		\end{aligned}\end{equation}
		The structure of \eqref{lindiff} is therefore the same as the linearized system \eqref{NSlin} and so the same estimates \eqref{mainest} from \cref{mainthm} hold. We deduce that
			\begin{equation}\label{xx1}\begin{aligned}
				\int_\Om\!\mathcal{E}(\tau_n,\boldsymbol{\omega}_n)(t)\,\dx+\int_0^t\int_\Om\!\left(|\Pt\boldsymbol{\omega}_n|^2+|\grad^2\boldsymbol{\omega}_n|^2\right)\dx\mathrm{d}s
				\le C(\e,\Omega,\widetilde{\rho},\widetilde{\bfu})\|\bfU_n-\bfU\|_{L^2_t(L^2_x)}^2,
			\end{aligned}\end{equation}
			where 
			\[\mathcal{E}(\tau_n,\boldsymbol{\omega}_n):= \frac{1}{2}\left(\tau_n^2+|\grad\tau_n|^2+|\boldsymbol{\omega}_n|^2+\mu|\grad\boldsymbol{\omega}_n|^2+(\mu+\lambda)|\diver\boldsymbol{\omega}_n|^2\right).\]
From the inequality \cref{ineqc}, we have $\bfU_n\rightarrow\bfU$ strongly in $L^2_t(L^2_x)$. Therefore $(\tau_n,\boldsymbol{\omega}_n)$ converges to $(0,0)$ strongly in the topologies determined by \eqref{xx1}. In particular, \eqref{xx0} holds.

To obtain the convergence of $\Pt\tau_n$, note that $\Pt\tau_n = -\diver(\tau_n\widetilde{\bfu})-\diver(\widetilde{\rho}\omega_n),$
and the right side of this equality converges to $0$ in $L^\infty_t(L^2_x)$.
	\end{proof}
Next we study limits of the following quantities,
	\begin{equation}\label{zvdef}
		z^h := \frac{\rho^h-\rho}{h},\quad \bfv^h := \frac{\bfu^h-\bfu}{h},
	\end{equation}
where $(\rho,\bfu)$ is the solution of \eqref{NSlin} corresponding to $\bfU$, and $(\rho^h,\bfu^h)$ is the solution corresponding
to the spike variation $\bfU^h$. We show that as $h\rightarrow 0$, the respective limits $z$ and $\bfv$ satisfy an appropriate linearized problem. 

Formally, one may directly write the system satisfied by $(z^h$, $\bfv^h)$, and taking the limit $h\rightarrow 0$ arrive at 
	\begin{equation}\label{bb5}\begin{aligned}
		&\Pt z + \diver(z\widetilde{\bfu})+\diver(\widetilde{\rho}\bfv)=0,\\
		&\Pt\bfv + \bfv\cdot\grad\widetilde{\bfu}+\widetilde{\bfu}\cdot\grad\bfv
		=(\widetilde{\rho})^{-1}[-\grad(z p'(\widetilde{\rho}))-z(\Pt\widetilde{\bfu}+\widetilde{\bfu}\cdot\grad\widetilde{\bfu})]\\
		&\qquad\qquad\qquad\qquad\qquad\,\,\,+(\widetilde{\rho})^{-1}[\diver\mathbb{S}(\grad\bfv)+z\bff]
			+(\widetilde{\rho})^{-1}(\bfW-\bfU)\delta(t-\tau),\\
		&z(0,\cdot)=0,\,\,\bfv(0,\cdot)=0,\\
		&\bfv(t,\cdot)\big|_{\Gamma}=0,
	\end{aligned}\end{equation}
where the singular term involving the Dirac delta corresponds to a jump due to the spike variation. Equivalently, the limiting system \eqref{bb5} may be cast in an abstract semigroup framework, where the singular term instead contributes as an initial datum 
at time $\tau$. These observations are made rigorous in the following theorem. Its proof is based on ideas from \cite{fs94,f99}. 
\begin{theorem}\label{convthm}
Suppose $(\rho,\bfu)$ is a solution of \eqref{NSlin} corresponding to the control $\bfU$, and let $(\rho^h,\bfu^h)$ be a solution corresponding to the spike variation $\bfU^h$, given in \cref{def}. Let $\tau\in (0,T]$ be a left Lebesgue point of the 
function $(\widetilde{\rho})^{-1}(t)(\bfW-\bfU(t))$. Define $z^h$ and $\bfv^h$ as in \eqref{zvdef}. Then $z^h\rightarrow z$
and $\bfv^h\rightarrow \bfv$ uniformly in $\tau\le t\le T$ and strongly in $L^2(\Om)$,
where
	\begin{equation}\label{evo}
	\left(\!\!\!\begin{array}{c}z\\ \bfv\end{array}\!\!\!\right)\!(t)=\left\{\begin{aligned}
	&(0,0)^T,\quad 0\le t< \tau,\\
	&S(t,\tau)\mathcal{F}_0(\tau),\quad \tau\le t\le T,
	\end{aligned}\right.\end{equation}
and 
	\[\mathcal{F}_0(\tau)=\left(\!\!\!\begin{array}{c}0\\ (\widetilde{\rho})^{-1}(\tau)(\bfW-\bfU(\tau))\end{array}\!\!\!\right),\]
and $S(t,\tau)$ is the evolution operator of the system 
	\begin{equation}\begin{aligned}\label{limiteq}
		&\Pt z + \diver(z\widetilde{\bfu})+\diver(\widetilde{\rho}\bfv)=0,\\
		&\Pt\bfv + \bfv\cdot\grad\widetilde{\bfu}+\widetilde{\bfu}\cdot\grad\bfv
		=(\widetilde{\rho})^{-1}[-\grad(z p'(\widetilde{\rho}))-z(\Pt\widetilde{\bfu}+\widetilde{\bfu}\cdot\grad\widetilde{\bfu})]
			+(\widetilde{\rho})^{-1}[\diver\mathbb{S}(\grad\bfv)+z\bff].
	\end{aligned}\end{equation}
\end{theorem}
\begin{proof}
Define a linear operator $\mathcal{A}(z,\bfv)^T=(\mathcal{A}_1(z,\bfv)^T,\mathcal{A}_2(z,\bfv)^T)$ where
	\begin{equation*}\begin{aligned}
	&-\mathcal{A}_1\!\left(\!\!\!\begin{array}{c}z\\ \bfv\end{array}\!\!\!\right) = \diver(z\widetilde{\bfu})+\diver(\widetilde{\rho}\bfv),\\
	&-\mathcal{A}_2\!\left(\!\!\!\begin{array}{c}z\\ \bfv\end{array}\!\!\!\right) = \bfv\cdot\grad\widetilde{\bfu}+\widetilde{\bfu}\cdot\grad\bfv
		+(\widetilde{\rho})^{-1}[z(\Pt\widetilde{\bfu}+\widetilde{\bfu}\cdot\grad\widetilde{\bfu})+\grad(z p'(\widetilde{\rho}))-\diver\mathbb{S}(\grad\bfv)-z\bff].
	\end{aligned}\end{equation*}
We may then concisely write the system for $(z^h,\bfv^h)^T$ as 
	\begin{equation*}\begin{aligned}
		&\frac{\mathrm{d}}{\dt}\!\left(\!\!\!\begin{array}{c}z^h\\ \bfv^h\end{array}\!\!\!\right) 
		= \mathcal{A}\!\left(\!\!\!\begin{array}{c}z^h\\ \bfv^h\end{array}\!\!\!\right)+\mathcal{F}_h(t),\\
		&\left(\!\!\!\begin{array}{c}z^h\\ \bfv^h\end{array}\!\!\!\right)\!(0) = \left(\!\!\!\begin{array}{c}0\\ 0\end{array}\!\!\!\right),
	\end{aligned}\end{equation*}
where 
	\[\mathcal{F}_h(t)=\left(0,\,(\widetilde{\rho})^{-1}(t)
	\frac{1}{h}\mathbbm{1}_{(\tau-h,\tau)}(t)(\bfW-\bfU(t))\right)^T.\]
Similarly, the system \eqref{limiteq} for $(z,\bfv)^T$ over $\tau\le t\le T$ may be written as
	\begin{equation*}\begin{aligned}
		&\frac{\mathrm{d}}{\dt}\!\left(\!\!\!\begin{array}{c}z\\ \bfv\end{array}\!\!\!\right) = \mathcal{A}\!\left(\!\!\!\begin{array}{c}z\\ \bfv\end{array}\!\!\!\right),\\
		&\left(\!\!\!\begin{array}{c}z\\ \bfv\end{array}\!\!\!\right)\!(\tau) 
		= \mathcal{F}_0(\tau),
	\end{aligned}\end{equation*}
while for $0\le t<\tau$ it holds that $(z,\bfv)^T(t)=(0,0)^T$. 

Next, define 
	\[\eta(t,h)=\left(\!\!\!\begin{array}{c}z^h\\ \bfv^h\end{array}\!\!\!\right)\!(t)-\left(\!\!\!\begin{array}{c}z\\ \bfv\end{array}\!\!\!\right)\!(t).\]
Note for any $0\le t<\tau-h$, $\mathcal{F}_h(t)=0$. It follows that for $t<\tau-h$,
	\begin{equation*}\begin{aligned}
		&\frac{\mathrm{d}}{\dt}\eta(t,h) = \mathcal{A}\eta(t,h),\\
		&\eta(0,h) = \left(\!\!\!\begin{array}{c}0\\ 0\end{array}\!\!\!\right).
	\end{aligned}\end{equation*}
Since $\mathcal{A}$ is linear, we deduce $\eta(t,h)=0$ uniformly in $t\in(0,\tau-h)$. Letting $h\rightarrow 0$ we deduce the convergence
for $0\le t< \tau$.

Now suppose $\tau\le t\le T$. By definition of the evolution operator $S(t,s)$,
	\[\left(\!\!\!\begin{array}{c}z^h\\ \bfv^h\end{array}\!\!\!\right)\!(t)=S(t,0)\!\left(\!\!\!\begin{array}{c}0\\ 0\end{array}\!\!\!\right)+\int_0^t\! S(t,s)\mathcal{F}_h(s)\,\mathrm{d}s=\int_0^t\! S(t,s)\mathcal{F}_h(s)\,\mathrm{d}s,\]
and
	\[\left(\!\!\!\begin{array}{c}z\\ \bfv\end{array}\!\!\!\right)\!(t)=S(t,\tau)\mathcal{F}_0(\tau).\]
Therefore
	\begin{equation*}\begin{aligned}
		\eta(t,h)&=\int_0^t\! S(t,s)\mathcal{F}_h(s)\,\mathrm{d}s-S(t,\tau)\mathcal{F}_0(\tau)\\
		&=\frac{1}{h}\int_{\tau-h}^\tau\! S(t,s)\mathcal{F}_0(s)-S(t,\tau)\mathcal{F}_0(\tau)\,\mathrm{d}s,
	\end{aligned}\end{equation*}
and so
	\begin{equation*}\begin{aligned}
		\|\eta(t,h)\|_{L^2_x\times L^2_x}
		&\le \frac{1}{h}\int_{\tau-h}^\tau\!\|S(t,s)\mathcal{F}_0(s)-S(t,\tau)\mathcal{F}_0(\tau)\|_{L^2_x\times L^2_x}\,\mathrm{d}s\\
		&\le \frac{1}{h}\int_{\tau-h}^\tau\!\|S(t,s)(\mathcal{F}_0(s)-\mathcal{F}_0(\tau))\|_{L^2_x\times L^2_x}
			+\|(S(t,s)-S(t,\tau))\mathcal{F}_0(\tau)\|_{L^2_x\times L^2_x}\,\mathrm{d}s\\
		&\le \frac{1}{h}\int_{\tau-h}^{\tau}\!\|S(t,s)\|_{\mathcal{L}(L^2_x\times L^2_x; L^2_x\times L^2_x)}
			\|\mathcal{F}_0(s)-\mathcal{F}_0(\tau)\|_{L^2_x\times L^2_x}\,\mathrm{d}s\\
		&\quad+\frac{1}{h}\int_{\tau-h}^{\tau}\!\|(S(t,s)-S(t,\tau))\mathcal{F}_0(\tau)\|_{L^2_x\times L^2_x}\,\mathrm{d}s.
	\end{aligned}\end{equation*}
As $h\rightarrow 0$, the right hand side of the above inequality converges to zero by virtue of strong continuity of the evolution operator in $L^2_x\times L^2_x$ and the left Lebesgue point property of $\mathcal{F}_0$ at $\tau$.
\end{proof}
Finally, the following theorem provides the existence of weak solutions of the adjoint system. It can be proven similarly to \cref{mainthm}.
\begin{theorem}\label{adjex}
	Let $\rho_d, \bfu_d \in L^2_t(L^2_x)$ be given and suppose $(\rho^*,\bfu^*,\bfU^*)$ is an optimal solution of Problem \cref{ptag}. Let
	$(\widetilde{\rho},\widetilde{\bfu})$ and $\bff$ satisfy the same conditions of \cref{mainthm}. Then there exists a weak solution $(\sigma,\bfxi)$ of the adjoint system \eqref{adj0} (in the sense of distributions), with regularity
	\begin{equation}
		\sigma\in L^\infty_t(L^2_x),
		\quad\bfxi\in L^\infty_t(L^2_x)\cap L^2_t(H^1_{0,x})
	\end{equation}
\end{theorem}
	
\section{Proof of the Pontryagin maximum principle}
Having collected all the preliminary results, we now prove \cref{mthm}. 

Let $(\rho^*,\bfu^*,\bfU^*)$ be an optimal triple for Problem \cref{ptag}. First we modify the cost functional by penalizing the state constraint. Following the choice of penalization in Wang and Wang \cite{ww03}, let $\e>0$ and define the \emph{penalized cost functional} 
\[\cost_\e(\rho,\bfu,\bfU) = \left[(\cost(\rho,\bfu,\bfU)-\cost(\rho^*,\bfu^*,\bfU^*)+\e)^2 + \mathrm{d}_W^2(F(\rho,\bfu))\right]^{1/2}\]
where $\mathrm{d}_W$ denotes the distance to the set $W\subset X$ in the norm of $X$, i.e.
	\[\mathrm{d}_W(F(\rho,\bfu)) = \inf_{w\in W}\|w-F(\rho,\bfu)\|_X.\]
We recall the state trajectory $(\rho,\bfu)(\cdot)$ is determined uniquely by the control $\bfU$ (cf. \cref{mainthm}), so that $\cost_\e$ may be considered a functional of $\bfU$ only.

Fix $\adcont$ as in \eqref{controls}. By \cref{comlem}, the metric space $(\adcont,d_E)$ is complete, where $d_E$ is the Ekeland metric. 
Recall that the maps $\bfU\mapsto \rho[\bfU]$ and $\bfU\mapsto\bfu[\bfU]$ are continuous in the topologies specified by \cref{nearby}, $\mathrm{d_W}$ is Lipschitz continuous with Lipschitz constant $1$, and $F$ is assumed to be Fr{\'e}chet differentiable. From these properties and the continuity condition \cref{ineqc}, it follows that $\cost_\e$ is lower semi-continuous (even continuous) with respect to the control $\bfU\in\adcont$. Furthermore, $\cost_\e$ is bounded from below, and the following inequalities hold,
	\begin{equation}
	\inf_{\bfU\in\adcont}\cost_\e(\rho,\bfu,\bfU)\le \cost_\e(\rho^*,\bfu^*,\bfU^*) = \e \le \inf_{\bfU\in\adcont}\cost_\e(\rho,\bfu,\bfU)+\e.
	\end{equation}
The triple $(\rho^*,\bfu^*,\bfU^*)$ is therefore an $\e$-minimizer of $\cost_\e$ and we may apply the Ekeland variational principle, \cref{th1}, with $\lambda=\sqrt{\e}$, to deduce there exists $(\rho_\e,\bfu_\e,\bfU_\e)\in\adall$ (cf. \cref{defall}) such that
	\begin{subequations}\begin{align}\label{ek0}
	& \cost_\e(\rho_\e,\bfu_\e,\bfU_\e)\le \cost_\e(\rho^*,\bfu^*,\bfU^*) = \e\\
	& d_E(\bfU_\e,\bfU^*)\le \sqrt{\e},\label{ek1}\\
	& \cost_\e(\rho,\bfu,\bfU)\ge \cost_\e(\rho_\e,\bfu_\e,\bfU_\e)-\sqrt{\e}d_E(\bfU,\bfU_\e),\qquad \forall\bfU\in\adcont.\label{ek03}
	\end{align}\end{subequations}
For the following computations we suppose $\e$ is fixed. The inequality \eqref{ek03} holds for any control in $\adcont$, and so we choose the admissible spike variation $\bfU_{\e}^h$ (see \cref{def}) of $\bfU_\e$ with corresponding trajectory $(\rho^h_\e,\bfu^h_\e)$ and deduce
	\begin{equation}\label{b1}\begin{aligned}
		-\sqrt{\e}&\le \frac{1}{h}\left[\cost_\e(\rho^h_\e,\bfu^h_\e,\bfU^h_\e)-\cost_\e(\rho_\e,\bfu_\e,\bfU_\e)\right]\\
		& = \frac{1}{h}\left[(\cost(\rho_\e^h,\bfu_\e^h,\bfU_\e^h)-\cost(\rho^*,\bfu^*,\bfU^*)+\e)^2 + \mathrm{d}_W^2(F(\rho_\e^h,\bfu_\e^h))\right]^{1/2}\\
		&\quad -\frac{1}{h}\left[(\cost(\rho_\e,\bfu_\e,\bfU_\e)-\cost(\rho^*,\bfu^*,\bfU^*)+\e)^2 + \mathrm{d}_W^2(F(\rho_\e,\bfu_\e))\right]^{1/2}
		\end{aligned}\end{equation}
From \cref{nearby} and the $L^2$-continuity provided by \cref{ineqc}, we observe that $\cost_\e(\rho_\e^h,\bfu_\e^h,\bfU_\e^h)=\cost_\e(\rho_\e,\bfu_\e,\bfU_\e)+o(1)$, and using the identity
	\[\sqrt{x} - \sqrt{y} = \frac{x-y}{\sqrt{x}+\sqrt{y}},\qquad x>0,\,y>0,\]
we obtain from \cref{b1} that
	\begin{equation}\label{kk}\begin{aligned}
	-\sqrt{\e}
	&\le \,
	\frac{1}{2\cost_\e(\rho_\e,\bfu_\e,\bfU_\e)+o(1)}
	\bigg\{\frac{1}{h}\bigg[(\cost(\rho_\e^h,\bfu_\e^h,\bfU_\e^h)
	-\cost(\rho^*,\bfu^*,\bfU^*)+\e)^2\\
	&\quad-(\cost(\rho_\e,\bfu_\e,\bfU_\e)-\cost(\rho^*,\bfu^*,\bfU^*)+\e)^2\bigg]
	+\frac{1}{h}[\mathrm{d}_W^2(F(\rho_\e^h,\bfu_\e^h))-\mathrm{d}_W^2(F(\rho_\e,\bfu_\e))]\bigg\}\\
	&= C_{\e}^h\frac{\cost(\rho_\e^h,\bfu_\e^h,\bfU_\e^h)-\cost(\rho_\e,\bfu_\e,\bfU_\e)}{h}\left(\cost(\rho_\e^h,\bfu_\e^h,\bfU_\e^h)+\cost(\rho_\e,\bfu_\e,\bfU_\e)-2\cost(\rho^*,\bfu^*,\bfU^*)+2\e\right)\\
	&\quad+C_\e^h\frac{\mathrm{d}_W^2(F(\rho_\e^h,\bfu_\e^h))-\mathrm{d}_W^2(F(\rho_\e,\bfu_\e))}{h},
	\end{aligned}\end{equation}
where $C_\e^h := (2\cost_\e(\rho_\e,\bfu_\e,\bfU_\e)+o(1))^{-1}$. Next we obtain the limit as $h\rightarrow 0$ in \cref{kk}. The computations are organized by first defining the following:
		\begin{equation}\begin{aligned}
		&\frac{1}{h}(\cost(\rho_\e^h,\bfu_\e^h,\bfU_\e^h)-\cost(\rho_\e,\bfu_\e,\bfU_\e)) + \frac{1}{h}(\mathrm{d}_W^2(F(\rho_\e^h,\bfu_\e^h))-\mathrm{d}_W^2(F(\rho_\e,\bfu_\e)))\\
		&=\frac{1}{2h}\int_0^T\!\|\bfu^h_\e-\bfu_d\|^2_{L^2_x}-\|\bfu_\e-\bfu_d\|^2_{L^2_x}\,\dt
			+\frac{1}{2h}\int_0^T\!\|\rho^h_\e-\rho_d\|^2_{L^2_x}-\|\rho_\e-\rho_d\|^2_{L^2_x}\,\dt\\
		&\quad+\frac{1}{2h}\int_0^T\!\|\bfU^h_\e\|^2_{L^2_x}-\|\bfU_\e\|^2_{L^2_x}\,\dt + \frac{1}{h}(\mathrm{d}_W^2(F(\rho_\e^h,\bfu_\e^h))-\mathrm{d}_W^2(F(\rho_\e,\bfu_\e)))\\
		&=: \sum_{j=1}^4 I_j^h,
	\end{aligned}\end{equation}
where the $I^h_j$ correspond to the preceding three integrals ($j=1,2,3$), and distance term ($j=4$).

By definition of the spike variation, on the time interval $(0,\tau-h)$, the control $\bfU_\e$ and variation $\bfU_\e^h$ coincide, and so
$(\rho_\e^h,\bfu_\e^h,\bfU_\e^h)\equiv (\rho_\e,\bfu_\e,\bfU_\e)$ on this time interval. Throughout this section we assume $\tau$ is taken to be a Lebesgue point of all relevant functions. Such a choice is always possible. Taking these considerations into account, we compute first
	\begin{equation}\label{b2}\begin{aligned}
		\lim_{h\rightarrow 0^+}I_1^h 
		&=\lim_{h\rightarrow 0^+}\frac{1}{h}\int_{\tau-h}^\tau\!\frac{1}{2}\left(\|\bfu^h_\e\|^2_{L^2_x}
		-\|\bfu_\e\|^2_{L^2_x}\right)-\left<\bfu^h_\e-\bfu_\e,\bfu_d\right>\,\dt\\
		&\quad+\lim_{h\rightarrow 0^+}\frac{1}{h}\int_\tau^T\!\frac{1}{2}\left(\|\bfu^h_\e\|^2_{L^2_x}
		-\|\bfu_\e\|^2_{L^2_x}\right)-\left<\bfu_\e^h-\bfu_\e,\bfu_d\right>\,\dt\\
		&=\lim_{h\rightarrow 0^+}\frac{1}{2}\int_\tau^T\!\left<\frac{\bfu^h_\e-\bfu_\e}{h},\bfu^h_\e+\bfu_\e-2\bfu_d\right>\,\dt.\\
	\end{aligned}\end{equation}
where the integral supported on $(\tau-h,\tau)$ vanishes in the limit since we choose $\tau$ to be a Lebesgue point.

Let $(z_\e,\bfv_\e)$ be a solution of the linearized system \eqref{limiteq} corresponding to the control $\bfU_\e$, as specified in \cref{convthm}.

By \cref{convthm} and \cref{nearby}, it follows that $h^{-1}(\bfu^h_\e-\bfu_\e)\rightarrow \bfv_\e$ in $L^\infty_t(L^2_x)$ as $h\rightarrow 0$ and $\bfu^h_\e\rightarrow \bfu_\e$ in $L^2_t(L^2_x)$ as $h\rightarrow 0$. Passing to the limit in \eqref{b2} we deduce
	\begin{equation}\label{limu}
		\lim_{h\rightarrow 0^+}I_1^h = \int_\tau^T\!\langle \bfu_\e-\bfu_d,\bfv_\e \rangle\,\dt.
	\end{equation}
Similarly, using that ${h}^{-1}(\rho^h_\e-\rho_\e)\rightarrow z_\e$ in $L^\infty_t(L^2_x)$ and $\rho^h_\e\rightarrow \rho_\e$ in $L^2_t(L^2_x)$ as $h\rightarrow 0$, we deduce
	\begin{equation*}\begin{aligned}
		\lim_{h\rightarrow 0^+}I_2^h 
		&=\lim_{h\rightarrow 0^+}\frac{1}{h}\int_{\tau-h}^\tau\!\frac{1}{2}\left(\|\rho^h_\e\|^2_{L^2_x}
		-\|\rho_\e\|^2_{L^2_x}\right)-\left<\rho^h_\e-\rho_\e,\rho_d\right>\,\dt\\
		&\quad+\lim_{h\rightarrow 0^+}\frac{1}{h}\int_\tau^T\!\frac{1}{2}\left(\|\rho^h_\e\|^2_{L^2_x}
		-\|\rho_\e\|^2_{L^2_x}\right)-\left<\rho^h_\e-\rho_\e,\rho_d\right>\,\dt\\
		&=\lim_{h\rightarrow 0^+}\frac{1}{2}\int_\tau^T\!\left<\frac{\rho^h_\e-\rho_\e}{h},\rho^h_\e+\rho_\e-2\rho_d\right>\,\dt\\
		&= \int_\tau^T\!\left<\rho_\e-\rho_d,z_\e\right>\,\dt
	\end{aligned}\end{equation*}
Next, using the definition of $\bfU^h_\e$,
	\begin{equation}\label{b4}\begin{aligned}
		\lim_{h\rightarrow 0^+}I_3^h &= 
		\lim_{h\rightarrow 0^+}\frac{1}{2h}\int_{\tau-h}^\tau\!\|\bfW\|^2_{L^2_x}-\|\bfU_\e(t)\|^2_{L^2_x}\,\dt\\
		& = \frac{1}{2}\|\bfW\|^2_{L^2_x}-\frac{1}{2}\|\bfU_\e(\tau)\|^2_{L^2_x}.
	\end{aligned}\end{equation}
Now we consider $I_4^h$. The squared distance function $w \mapsto \mathrm{d}^2_W(w)$ is continuously Fr{\'e}chet differentiable on $X$ with 
	\begin{equation*}
	D\mathrm{d}^2_W(w) = 
	\left\{\begin{aligned}
	& 2\mathrm{d}_W(w)\eta,\qquad \{\eta\}=\partial\mathrm{d}_W(w),\,\,\mathrm{if}\,\, w\notin W\\
	& 0,\qquad\qquad\qquad\mathrm{if}\,\,w\in W.
	\end{aligned}\right.
	\end{equation*}
By assumption, $X^*$ is strictly convex, and so if $w\notin W$, $\partial\mathrm{d}_W(w)$ consists of a single element with unit norm in $X^*$ (cf. page 154, Li and Yong \cite{ly95optimal}). Hence without loss of generality we can write
	\begin{equation}\label{distdiv}\left\{\begin{aligned}
	& D\mathrm{d}^2_W(w) = 2\mathrm{d}_W(w)\eta,\\
	& \eta\in \partial\mathrm{d}_W(w),\qquad \|\eta\|_{X^*}=1.
	\end{aligned}\right.
	\end{equation}
Furthermore, we have the following Fr{\'e}chet derivative of $F$ at the point $(\rho,\bfu)$ in terms of its partial Fr{\'e}chet derivatives (cf. Proposition 2.53 \cite{penot}):
	\begin{equation}\label{pard}
	[DF(\rho,\bfu)](z,{\bf v}) = [F_\rho(\rho,\bfu)]z + [F_{\bfu}(\rho,\bfu)]{\bf v}.
	\end{equation}
From \cref{convthm} we have that
	\begin{equation}\label{taylor}\begin{aligned}
		&\rho_\e^h = \rho_\e + hz_\e + hr_{1}^h,\qquad\quad
			\lim_{h\rightarrow 0^+}{\|r_{1}^h\|_{C([0,T];L^2(\Omega))} = 0}\\
		&\bfu_\e^h = \bfu_\e + h{\bf v}_\e + hr_{2}^h,\qquad\quad 
			\lim_{h\rightarrow 0^+}{\|r_{2}^h\|_{C([0,T];L^2(\Omega;\mathbb{R}^3))} = 0}\\
	\end{aligned}\end{equation}
The composition of Fr{\'e}chet differentiable functions is differentiable and obeys a chain rule, and so we obtain from \cref{distdiv}-\cref{taylor} that
	\begin{equation}\begin{aligned}\label{sublim1}
		\lim_{h\rightarrow 0^+} I_4^h
		&= \lim_{h\rightarrow 0^+}\frac{1}{h}(\mathrm{d}_W^2(F(\rho_\e^h,\bfu_\e^h))-\mathrm{d}_W^2(F(\rho_\e,\bfu_\e)))\\
		&= \langle 2\mathrm{d}_W(F(\rho_\e,\bfu_\e))\eta_\e,[F_\rho(\rho_\e,\bfu_\e)]z_\e + [F_{\bfu}(\rho_\e,\bfu_\e)]{\bf v_\e}\rangle_{X^*;X},
	\end{aligned}\end{equation}
where $\eta_\e\in \partial\mathrm{d}_W(F(\rho_\e,\bfu_\e))\subset X^*$ and $\|\eta_\e\|_{X^*} = 1$.

We are now in a position to let $h\rightarrow 0$ in \cref{kk}, obtaining
	\begin{equation}\begin{aligned}\label{epad}
		-\sqrt{\e}\le 
		&\int_\tau^T\!\left<\lambda_\e(\bfu_\e-\bfu_d),
		\bfv_\e\right>+\left<\lambda_\e(\rho_\e-\rho_d),z_\e\right>\,\dt
		+\frac{1}{2}\lambda_\e(\|\bfW\|^2_{L^2_x} -\|\bfU_\e(\tau)\|^2_{L^2_x})\\
		&+\langle a_\e,[F_\rho(\rho_\e,\bfu_\e)]z_\e+[F_\bfu(\rho_\e,\bfu_\e)]\bfv_\e \rangle_{X^*;X},
	\end{aligned}\end{equation}
where 
	\begin{equation*}
		a_\e := \frac{\mathrm{d}_W(F(\rho_\e,\bfu_\e))}{\cost_\e(\rho_\e,\bfu_\e,\bfU_\e)}\eta_\e,\qquad\quad \lambda_\e := \frac{\cost(\rho_\e,\bfu_\e,\bfU_\e)-\cost(\rho^*,\bfu^*,\bfU^*)+\e}{\cost_\e(\rho_\e,\bfu_\e,\bfU_\e)}
	\end{equation*}
Next we introduce the Hamiltonian. From the weak formulation of the adjoint equations (cf. \cref{adjex}) we let $(\sigma_\e,\bfxi_\e)$ be a weak solution of 
\begin{equation}\label{adj1}\begin{aligned}
		&-\Pt\sigma_\e-\widetilde{\bfu}\cdot\grad\sigma_\e
		=(\widetilde{\rho})^{-1}\bfxi_\e\cdot[-(\Pt\widetilde{\bfu}+\widetilde{\bfu}\cdot\grad\widetilde{\bfu})+\bff] +p'(\widetilde{\rho})\diver\left((\widetilde{\rho})^{-1}\bfxi_\e\right)\\
		&\hspace{80pt}+\lambda_\e(\rho_d-\rho_\e)-[F_\rho(\rho_\e,\bfu_\e)]^*a_\e,\\
		&-\Pt\bfxi_\e-\diver(\bfxi_\e\otimes\widetilde{\bfu})+\bfxi_\e\cdot(\grad\widetilde{\bfu})^T =\widetilde{\rho}\grad\sigma_\e+\diver\mathbb{S}(\grad((\widetilde{\rho})^{-1}\bfxi_\e))+\lambda_\e(\bfu_d-\bfu_\e)-[F_\bfu(\rho_\e,\bfu_\e)]^*a_\e,\\
		&\bfxi_\e\big|_\Gamma = 0,\\
		&\sigma_\e(T,\cdot)=0,\quad\bfxi_\e(T,\cdot)=0,
	\end{aligned}\end{equation}
Combining with \eqref{epad} we get,
	\begin{equation}\label{end}\begin{aligned}
		-\sqrt{\e}
		&\le \int_\tau^T\!\left<\sigma_\e,-\Pt z_\e-\diver(z_\e \widetilde{\bfu})-\diver(\widetilde{\rho}\bfv_\e)\right>\dt\\
		&\quad +\int_\tau^T\!\left<\bfxi_\e,-(\widetilde{\rho})^{-1}z_\e(\Pt\widetilde{\bfu}+\widetilde{\bfu}\cdot\grad\widetilde{\bfu})
			-(\widetilde{\rho})^{-1}\grad(z_\e p'(\widetilde{\rho}))\right>\dt\\
		&\quad +\int_\tau^T\!\left<\bfxi_\e,(\widetilde{\rho})^{-1}z_\e\bff+(\widetilde{\rho})^{-1}\diver\mathbb{S}(\grad\bfv_\e)
		-\Pt\bfv_\e-\widetilde{\bfu}\cdot\grad\bfv_\e-\bfv_\e\cdot\grad\widetilde{\bfu}\right>\dt\\
		&\quad +\frac{1}{2}\lambda_\e(\|\bfW\|^2_{L^2_x}-\|\bfU_\e(\tau)\|^2_{L^2_x})
		+\int_\Om\!\sigma_\e(T)z_\e(T)-\sigma_\e(\tau)z_\e(\tau)\,\dx\\
		&\quad+\int_\Om\!(\widetilde{\rho})^{-1}(T)\bfxi_\e(T)\cdot\bfv_\e(T)-
			(\widetilde{\rho})^{-1}(\tau)\bfxi_\e(\tau)\cdot\bfv_\e(\tau)\,\dx\\
		&=\frac{1}{2}\lambda_\e(\|\bfW\|^2_{L^2_x}-\|\bfU_\e(\tau)\|^2_{L^2_x})-\int_\Om\!(\widetilde{\rho})^{-1}(\tau)\bfxi_\e(\tau)\cdot
			(\bfW-\bfU_\e(\tau))\,\dx
	\end{aligned}\end{equation}
valid for all $\bfW\in\overline{B}_R(0)$. The strong solution property of $(z_\e,\bfv_\e)$ allowed to remove the time integrals. In \eqref{end} we also use that $\sigma_\e(T)=\bfxi_\e(T)=0$ and 
$z_\e(\tau)=0$, $\bfv_\e(\tau)=(\widetilde{\rho})^{-1}(\tau)(\bfW-\bfU_\e(\tau))$. The adjoint equations \eqref{adj1} and the inequality \eqref{end} may be interpreted as necessary conditions for $\e$-optimal control.

To conclude, we must pass $\e\rightarrow 0$ in \cref{end}. From the definition of $a_\e$ and $\lambda_\e$, and using that $\|\eta_\e\|_{X^*}=1$, it follows that
	\begin{equation}\label{bounds}
		1\le\lambda_\e + \|a_\e\|_{X^*}\le 2.
	\end{equation}
Therefore, there exist $\lambda\in \mathbb{R}$ and $a\in X^*$ such that (along subsequences),
	\begin{equation}\label{etac}
	\lambda_\e \rightarrow \lambda\quad\mathrm{as}\,\,\e\rightarrow 0,
	\end{equation}
and
	\begin{equation}\label{etaco}
	a_\e\weak^* a\quad\mathrm{in}\,\, X^*\,\,\mathrm{as}\,\, \e\rightarrow 0.
	\end{equation}
From the estimates on $\bfxi_\e$ using \cref{adjex}, we further obtain that $\Pt\bfxi_\e \in L^2_t(H^{-1}_x)$ uniformly in $\e$. From this estimate and $\bfxi_\e \in L^2_t(H^1_{0,x})$ we obtain by continuous embedding that $\bfxi_\e \in C([0,T];L^2(\Om;\mathbb{R}^3))$. Furthermore, by an application of the Aubin-Lions lemma, we obtain that
	\begin{equation}\label{exiconv}	
	\bfxi_\e\rightarrow \bfxi\quad\mathrm{strongly\,\,in\,\,}L^2_t(L^2_x)\,\,\mathrm{as\,\,}\e\rightarrow 0.
	\end{equation}
Furthermore, from \cref{ek1} we obtain convergence of the control terms,
	\begin{equation}\label{limuu}
	\bfU_\e\rightarrow \bfU^*\quad \mathrm{strongly\,\,in}\,\,L^2_t(L^2_x)\,\,\mathrm{as\,\,}\e\rightarrow 0.
	\end{equation} 
Integrating \cref{end} in time from $0$ to $T$ and using \cref{etac},  \cref{exiconv}, and \cref{limuu}, we pass $\e\rightarrow 0$ obtaining
	\begin{equation}\label{intlim}
	0\le \frac{1}{2}\lambda\int_0^T\int_\Omega\! |\bfW|^2-|\bfU^*|^2 \,\dxdt
	- \int_0^T\int_\Omega\! (\widetilde{\rho})^{-1}\bfxi\cdot(\bfW-\bfU^*) \,\dxdt
	\end{equation}
From the weak star convergence \eqref{etaco} and using that $F$ is continuously Fr{\'e}chet differentiable, we furthermore obtain that
	\begin{equation}\label{weaks}\begin{aligned}
		&[F_\rho(\rho_\e,\bfu_\e)]^*a_\e \weak [F_\rho(\rho,\bfu)]^*a\quad\mathrm{weakly\,\,in\,\,}L^2_t(L^2_x),\\
		&[F_\bfu(\rho_\e,\bfu_\e)]^*a_\e \weak [F_\bfu(\rho,\bfu)]^*a\quad\mathrm{weakly\,\,in\,\,}L^2_t(H^{-1}_x).
	\end{aligned}\end{equation}
Passing to the limit $\e\rightarrow 0$ in \eqref{adj1}, we arrive at 
	\begin{equation}\label{elim}\begin{aligned}
	&-\Pt\sigma-\widetilde{\bfu}\cdot\grad\sigma
		=(\widetilde{\rho})^{-1}\bfxi\cdot[-(\Pt\widetilde{\bfu}+\widetilde{\bfu}\cdot\grad\widetilde{\bfu})+\bff] +p'(\widetilde{\rho})\diver\left((\widetilde{\rho})^{-1}\bfxi\right)\\
		&\hspace{80pt}+\lambda(\rho_d-\rho^*)-[F_\rho(\rho^*,\bfu^*)]^*a,\\
		&-\Pt\bfxi-\diver(\bfxi\otimes\widetilde{\bfu})+\bfxi\cdot(\grad\widetilde{\bfu})^T =\widetilde{\rho}\grad\sigma+\diver\mathbb{S}(\grad((\widetilde{\rho})^{-1}\bfxi))+\lambda(\bfu_d-\bfu^*)-[F_\bfu(\rho^*,\bfu^*)]^*a,\\
		&\bfxi\big|_\Gamma = 0,\\
		&\sigma(T,\cdot)=0,\,\,\bfxi(T,\cdot)=0.\\
	\end{aligned}\end{equation}
The integral maximum principle and adjoint equations for $\bfU^*$ to be an optimal control have been obtained. 

Furthermore,
using that $a_\e\in \partial \mathrm{d}_W(F(\rho_\e,\bfu_\e))$, by the definition of subdifferential we get that
	\[\mathrm{d}_W(w)\ge \mathrm{d}_W(F(\rho_\e,\bfu_\e))+\langle a_\e,w-F(\rho_\e,\bfu_\e) \rangle_{X^*;X}\quad\forall w\in W.\]
Using that $\mathrm{d}_W(w)=0$ and the nonnegativity of the distance function, it follows that
\[\langle a_\e, w-F(\rho_\e,\bfu_\e) \rangle\le 0.\]
Passing $\e\rightarrow 0$, and using the weak convergence of $a_\e$ and the strong convergences of $\rho_\e$ and $\bfu_\e$, it follows that
	\begin{equation}\label{ncone}
		\langle \eta,w-F(\rho^*,\bfu^*) \rangle \le 0\quad\forall w\in W.
	\end{equation}
The condition \eqref{ncone} says that $a$ belongs to the normal cone of $W$ at $F(\rho^*,\bfu^*)$, i.e. $\eta_0\in N_W(F(\rho^*,\bfu^*))$.

This concludes the proof of \cref{mthm}. It only remains to justify the existence of an optimal triple $(\rho^*,\bfu^*,\bfU^*)$.

\section{Existence of optimal controls}
In this section we establish the existence of optimal controls for Problem \eqref{ptag}. We take as an assumption that the set of admissible triples $\adall$ is nonempty.
	\begin{theorem}\label{opex}
		Let $\cost$ be defined by \eqref{cost}, with $\rho_d, \bfu_d\in L^2_t(L^2_x)$ given. Let $\adall$ be defined as in 
		\cref{defall}, and suppose $\adall\neq\emptyset$. Then there exists an optimal triple $(\rho^*,\bfu^*,\bfU^*)\in\adall$ such that
			\[\cost(\rho^*,\bfu^*,\bfU^*)=\inf_{(\rho,\bfu,\bfU)\in\adall}\cost(\rho,\bfu,\bfU)=: j.\]
	\end{theorem}
\begin{proof}
We employ the direct method from the calculus of variations. By assumption, $\adall$ is nonempty. Since $\cost$ is bounded below, we deduce the existence of a  minimizing sequence $\{(\rho^n,\bfu^n,\bfU^n)\}_{n=1}^\infty$, of elements of $\adall$, such that 
	\begin{equation}\label{ex1}
		\lim_{n\rightarrow\infty}\cost(\rho^n,\bfu^n,\bfU^n)= j.
	\end{equation}
Furthermore, there exists $R$ large enough such that $0\le \cost(\rho^n,\bfu^n,\bfU^n)\le R<+\infty$, uniformly in $n$. In particular, 
$\|\bfU^n\|_{L^2_t(L^2_x)}\le C(R)$, and by Theorem \ref{mainthm}, we obtain  estimates on 
$\rho^n\in L^\infty_t(H^1_x)$, $\Pt\rho^n\in L^\infty_t(L^2_x)$, $\bfu^n\in L^\infty_t(H^1_{0,x})\cap L^2_t(H^2_x)$,
and $\Pt\bfu^n\in L^2_t(L^2_x)$, uniform in $n$. 

Without relabeling, there exists a subsequence $(\rho^n,\bfu^n,\bfU^n)$ converging 
weakly to a triple $(\rho^*,\bfu^*,\bfU^*)$ in $[L^2_t(L^2_x)]^3$. Recall $\adcont$ is a closed and convex set.
Since $\cost$ is continuous and convex over $L^2_t(L^2_x)\times L^2_t(L^2_x)\times\adcont$, it follows that $\cost$ is 
also sequentially weakly lower semi-continuous. We deduce that 
	\begin{equation}\label{ex2}
		\cost(\rho^*,\bfu^*,\bfU^*)\le \liminf_{n\rightarrow\infty}\cost(\rho^n,\bfu^n,\bfU^n).
	\end{equation}
Combining \eqref{ex1} and \eqref{ex2} we get that
	\begin{equation*}
		j\le \cost(\rho^*,\bfu^*,\bfU^*)
		\le \liminf_{n\rightarrow\infty}\cost(\rho^n,\bfu^n,\bfU^n)
		= \lim_{n\rightarrow\infty}\cost(\rho^n,\bfu^n,\bfU^n) = j.
	\end{equation*}
Therefore $(\rho^*,\bfu^*,\bfU^*)$ is a minimizer. 

It remains to check that 
$(\rho^*,\bfu^*,\bfU^*)$ is a strong solution of \eqref{NSlin} and satisfies the state constraint. However, from the uniform estimates, we also obtain along 
the subsequence that 
$\rho^n\weak^* \rho^*$ in $L^\infty_t(H^1_x)$, $\Pt\rho^n\weak^* \Pt\rho^*$ in $L^\infty_t(L^2_x)$, $\bfu^n\weak^* \bfu^*$ in 
$L^\infty_t(H^1_{0,x})\cap L^2_t(H^2_x)$, and $\Pt\bfu^n\weak \Pt\bfu^*$ in $L^2_t(L^2_x)$. Since the equations \eqref{NSlin} are linear, we may pass to the limit to conclude that $(\rho^*,\bfu^*,\bfU^*)$ satisfies the PDE. In particular, we may pass to the limit in \eqref{NSlin} weakly in $L^2_t(L^2_x)$ and use the density of test functions to conclude the governing equations are satisfied almost everywhere in $Q_T$.

Furthermore, the convergence on $(\rho^n,\bfu^n)$ allow us to conclude by Aubin-Lions lemma that along subsequences $(\rho^n,\bfu^n)$ converges strongly to $(\rho^*,\bfu^*)$ in $L^2_t(L^2_x)\times L^2_t(H^1_{0,x})$. By continuity of $F$ through its Frechet differentiability and continuity of the distance function it follows that
	\[\lim_{n\rightarrow\infty}\mathrm{d}_W(F(\rho^n,\bfu^n)) = \mathrm{d}_W(F(\rho^*,\bfu^*)),\]
and so $(\rho^*,\bfu^*)\in W$ since $W$ is closed in $X$.

The initial conditions make sense noting that
$\rho^n\in L^2_t(H^1_x)$ and $\Pt\rho^n\in L^2_t(L^2_x)$ imply $\rho^n\rightarrow \rho^*$ in $C([0,T];L^2(\Om))$ along a subsequence by the Aubin-Lions lemma. Considering the weak formulations for the admissible pair $(\rho^n,\bfu^n)$ and $(\rho^*,\bfu^*)$, and an appropriate choice of test functions, it follows that $\langle \rho^*(0) - \rho_0 , \phi \rangle=0$ for all $\phi\in L^2_x$ which implies $\rho^*(0)=\rho_0$. A similar argument applies to the velocity in order to obtain $\bfu^*(0)=\bfu_0$. 
\end{proof}

\begin{appendix}
\section{Proof of Theorem \ref{mainthm}}
\subsection{A priori estimates}
The \emph{a priori} estimates are organized into the following propositions.
\begin{proposition}\label{prop1}
Suppose the assumptions of Theorem \ref{mainthm} are satisfied. Then any regular (smooth) solution $(\rho,\bfu)$ of the linearized system \eqref{NSlin} satisfies the following energy inequality for all $t\in (0,T)$:
	\begin{equation}\label{en1}\begin{aligned}
		&\int_\Om\frac{1}{2}\left(|\rho|^2+|\grad\rho|^2+|\bfu|^2\right)\!(t,x)\,\dx
		+\int_0^t\int_\Om\!\left(\mu|\grad\bfu|^2+(\mu+\lambda)|\diver\bfu|^2\right)\dx\mathrm{d}s\\
		&\lesssim \int_\Om\!\frac{1}{2}\left(|\rho_0|^2+|\grad\rho_0|^2+|\bfu_0|^2\right)\dx
		+\int_0^t\! A(s)\int_\Om\!\frac{1}{2}\left(|\rho|^2+|\grad\rho|^2+|\bfu|^2\right)\dx\mathrm{d}s\\
		&\quad+\e\|\widetilde{\rho}\|_{L^\infty_t(L^\infty_x)}\int_0^t\int_\Om\!|\grad\diver\bfu|^2\,\dx\mathrm{d}s
		+C(\e)\int_0^t\int_\Om\!|\bfU|^2\,\dx\mathrm{d}s,
	\end{aligned}\end{equation}
where $A(\cdot)\in L^1(0,T)$ depends on $\e>0$ small enough, $\|\grad\widetilde{\bfu}\|_{L^\infty_x}$, 
$\|\Pt\widetilde{\bfu}+\widetilde{\bfu}\cdot\grad\widetilde{\bfu}\|_{L^{3/2}_x}^2$, $\|\bff\|_{L^{3/2}_x}^2$,
$\|p'(\widetilde{\rho})\|_{L^\infty_t(L^\infty_x)}$, $\|\widetilde{\rho}\|_{L^\infty_t(L^\infty_x)}$,
$\|\grad\diver\widetilde{\bfu}\|_{L^3_x}^2$, $\|\grad\widetilde{\rho}\|_{L^\infty_x}^2$, and
$\|\grad^2\widetilde{\rho}\|_{L^3_x}^2$.
\end{proposition}
\begin{proof}
Multiplying \eqref{l1} by $\rho$, integrating by parts over $\Om$, and noting that the velocity $\widetilde{\bfu}$ vanishes on the boundary, we deduce
	\begin{equation}\label{ii0}\begin{aligned}
		\frac{\mathrm{d}}{\dt}\int_\Om\!\frac{1}{2}\rho^2\,\dx = -\int_\Om\!\frac{1}{2}\rho^2 \diver\widetilde{\bfu}\,\dx
			-\int_\Om\! \rho(\widetilde{\rho}\diver\bfu+\bfu\cdot\grad\widetilde{\rho})\,\dx.
	\end{aligned}\end{equation}	
Next, using that $(\widetilde{\rho},\widetilde{\bfu})$ satisfies equation \eqref{l02}, we get
	\[\widetilde{\rho}(\Pt\bfu+\widetilde{\bfu}\cdot\grad\bfu)=\Pt(\widetilde{\rho}\bfu)+\diver(\widetilde{\rho}\bfu\otimes\widetilde{\bfu}),\]
and by a simple computation
	\begin{equation}\label{i}
		\bfu\cdot[\Pt(\widetilde{\rho}\bfu)+\diver(\widetilde{\rho}\bfu\otimes\widetilde{\bfu})]
		=\Pt\left(\frac{1}{2}\widetilde{\rho}|\bfu|^2\right)+\diver\left(\frac{1}{2}\widetilde{\rho}|\bfu|^2\widetilde{\bfu}\right).
	\end{equation}
Taking the scalar product of the momentum equation \eqref{l2} with $\bfu$, invoking \eqref{i}, and integrating by parts we deduce
	\begin{equation}\label{ii2}\begin{aligned}
		\frac{\mathrm{d}}{\dt}\int_\Om\!\frac{1}{2}\widetilde{\rho}|\bfu|^2\,\dx+\int_\Om\!\mathbb{S}(\grad\bfu):\grad\bfu\,\dx
		&= -\int_\Om\!\widetilde{\rho}\bfu\cdot(\bfu\cdot\grad)\widetilde{\bfu}\,\dx
			-\int_\Om\!\rho(\Pt\widetilde{\bfu}+\widetilde{\bfu}\cdot\grad\widetilde{\bfu})\cdot\bfu\,\dx\\
		&\quad+\int_\Om\!\left(\rho\bfu\cdot\bff+\bfu\cdot\bfU\right)\dx
		+\int_\Om\!\rho p'(\widetilde{\rho})\diver\bfu\,\dx.
	\end{aligned}\end{equation}
Next we obtain an estimate on the density gradient. Applying the gradient operator to \eqref{l1} and taking the scalar product 
with $\grad\rho$, we get
	\begin{equation}\label{ii3}\begin{aligned}
		\frac{\mathrm{d}}{\dt}\int_\Om\!\frac{1}{2}|\grad\rho|^2\,\dx
		&= -\int_\Om\!\grad\rho\cdot[\grad\diver(\rho\widetilde{\bfu})+\grad\diver(\widetilde{\rho}\bfu)]\,\dx\\
		&=-\int_\Om\!\left(\frac{1}{2}\diver\widetilde{\bfu} |\grad\rho|^2
		+\rho\grad\rho\cdot\grad\diver\widetilde{\bfu}+\grad\rho\otimes\grad\rho:\grad\widetilde{\bfu}\right)\dx\\
		&\quad-\int_\Om\!\left((\grad\rho\cdot\grad\widetilde{\rho})\diver\bfu+\widetilde{\rho}\grad\rho\cdot\grad\diver\bfu\right)\dx\\
		&\quad-\int_\Om\!\left(\grad\rho\otimes\bfu:\grad^2\widetilde{\rho}+\grad\widetilde{\rho}\otimes\grad\rho:\grad\bfu\right)\dx,
	\end{aligned}\end{equation}
where the second equality follows from a few applications of the product rule and integrating by parts.

Finally, by virtue of $\bfu$ vanishing on $\Gamma_T$,
	\begin{equation}\label{seq}
		\int_\Om\!\mathbb{S}(\grad\bfu):\grad\bfu\,\dx = \int_\Om\!\left(\mu|\grad\bfu|^2+(\mu+\lambda)|\diver\bfu|^2\right)\dx.
	\end{equation}
Combining \eqref{ii0}, \eqref{ii2}, \eqref{ii3} and \eqref{seq}, and integrating in time we arrive at the energy identity
	\begin{equation}\label{iienergy}\begin{aligned}
		&\int_\Om\!\frac{1}{2}\left(\rho^2+|\grad\rho|^2+\widetilde{\rho}|\bfu|^2\right)\!(t,x)\,\dx
		+\int_0^t\int_\Om\!\left(\mu|\grad\bfu|^2+(\mu+\lambda)|\diver\bfu|^2\right)\dx\mathrm{d}s\\
		&= \int_\Om\!\frac{1}{2}\left(\rho_0^2+|\grad\rho_0|^2+\widetilde{\rho_0}|\bfu_0|^2\right)\dx
		-\int_0^t\int_\Om\!\widetilde{\rho}\bfu\cdot(\bfu\cdot\grad)\widetilde{\bfu}\,\dx\mathrm{d}s\\
		&\quad-\int_0^t\int_\Om\!\rho(\Pt\widetilde{\bfu}+\widetilde{\bfu}\cdot\grad\widetilde{\bfu})\cdot\bfu\,\dx\mathrm{d}s
			+\int_0^t\int_\Om\!\left(\rho\bfu\cdot\bff+\bfu\cdot\bfU\right)\dx\mathrm{d}s\\
		&\quad+\int_0^t\int_\Om\!\rho p'(\widetilde{\rho})\diver\bfu\,\dx\mathrm{d}s
			-\int_0^t\int_\Om\!\frac{1}{2}\rho^2\diver\widetilde{\bfu}\,\dx\mathrm{d}s
			-\int_0^t\int_\Om\!\rho\widetilde{\rho}\diver\bfu\,\dx\mathrm{d}s\\
		&\quad-\int_0^t\int_\Om\!\rho\bfu\cdot\grad\widetilde{\rho}\,\dx\mathrm{d}s
		-\int_0^t\int_\Om\!\frac{1}{2}\diver\widetilde{\bfu} |\grad\rho|^2\,\dx\mathrm{d}s
		-\int_0^t\int_\Om\!\rho\grad\rho\cdot\grad\diver\widetilde{\bfu}\,\dx\mathrm{d}s\\
		&\quad-\int_0^t\int_\Om\!\grad\rho\otimes\grad\rho:\grad\widetilde{\bfu}\,\dx\mathrm{d}s
		-\int_0^t\int_\Om\!(\grad\rho\cdot\grad\widetilde{\rho})\diver\bfu\,\dx\mathrm{d}s\\
		&\quad-\int_0^t\int_\Om\!\widetilde{\rho}\grad\rho\cdot\grad\diver\bfu\,\dx\mathrm{d}s
		-\int_0^t\int_\Om\!\grad\rho\otimes\bfu:\grad^2\widetilde{\rho}\,\dx\mathrm{d}s\\
		&\quad-\int_0^t\int_\Om\!\grad\widetilde{\rho}\otimes\grad\rho:\grad\bfu\,\dx\mathrm{d}s\\
		&= \int_\Om\!\frac{1}{2}\left(\rho_0^2+|\grad\rho_0|^2+\widetilde{\rho_0}|\bfu_0|^2\right)\dx
			+\sum_{i=1}^{14} I_i,
	\end{aligned}\end{equation}
where each of the $I_i$ denote one of the space-time integrals.
Repeatedly invoking H{\"o}lder's inequality, Young's inequality with $\e$, Poincar{\'e}'s inequality, and the Sobolev embedding 
$W^{1,2}_x\subset L^6_x$, we next estimate each of the $I_i$ as follows:
	\begin{align*}\label{bb}
		&|I_1|
		\le \int_0^t\!2\|\grad\widetilde{\bfu}(s,\cdot)\|_{L^\infty_x}\int_\Om\!\frac{1}{2}\widetilde{\rho}|\bfu|^2\,\dx\mathrm{d}s,\\
		&|I_2|
		\le C(\e)\int_0^t\!\|(\Pt\widetilde{\bfu}+\widetilde{\bfu}\cdot\grad\widetilde{\bfu})(s,\cdot)\|^2_{L^{3/2}_x}\|\rho\|^2_{W^{1,2}_x}
		\,\mathrm{d}s
			+\e\int_0^t\int_\Om\!|\grad\bfu|^2\,\dx\mathrm{d}s,\\
		&|I_3| \le C(\e)\int_0^t\!\|\bff(s,\cdot)\|^2_{L^{3/2}_x}\|\rho\|^2_{W^{1,2}_x}\,\mathrm{d}s
			+\e\int_0^t\int_\Om\!|\grad\bfu|^2\,\dx\mathrm{d}s\\
		&\qquad + C(\e)\int_0^t\int_\Om\!|\bfU|^2\,\dx\mathrm{d}s,\\
		&|I_4| \le \|p'(\widetilde{\rho})\|_{L^\infty_t(L^\infty_x)}
		\left(C(\e)\int_0^t\int_\Om\!\frac{1}{2}\rho^2\,\dx\mathrm{d}s+\e\int_0^t\int_\Om\!|\grad\bfu|^2\,\dx\mathrm{d}s\right),\\
		&|I_5| \le \int_0^t\|\diver\widetilde{\bfu}(s,\cdot)\|_{L^\infty_x}\int_\Om\!\frac{1}{2}\rho^2\,\dx\mathrm{d}s,\\
		&|I_6|\le \|\widetilde{\rho}\|_{L^\infty_t(L^\infty_x)}
		\left(C(\e)\int_0^t\int_\Om\!\frac{1}{2}\rho^2\,\dx\mathrm{d}s+\e\int_0^t\int_\Om\!|\grad\bfu|^2\,\dx\mathrm{d}s\right),\\
		&|I_7|
		\le C(\e)\int_0^t\!\|\grad\widetilde{\rho}(s,\cdot)\|^2_{L^{\infty}_x}\|\rho\|^2_{L^2_x}\,\mathrm{d}s
			+\e\int_0^t\int_\Om\!|\grad\bfu|^2\,\dx\mathrm{d}s,\\
		&|I_8|\le \frac{1}{2}\int_0^t\!\|\diver\widetilde{\bfu}(s,\cdot)\|_{L^\infty_x}\int_\Om\!|\grad\rho|^2\,\dx\mathrm{d}s,\\
		&|I_9|\le C\int_0^t\!\|\grad\diver\widetilde{\bfu}(s,\cdot)\|^2_{L^3_x}\|\rho\|^2_{W^{1,2}_x}
		\,\mathrm{d}s+\frac{1}{2}\int_0^t\int_\Om\!|\grad\rho|^2\,\dx\mathrm{d}s,\\
		&|I_{10}|\le \int_0^t\!\|\grad\widetilde{\bfu}(s,\cdot)\|_{L^\infty_x}\int_\Om\!|\grad\rho|^2\,\dx\mathrm{d}s,\\
		&|I_{11}|\le \int_0^t\!\|\grad\widetilde{\rho}(s,\cdot)\|_{L^\infty_x}
		\left(\frac{1}{2}\int_\Om\!|\grad\rho|^2\,\dx+C\int_\Om\!|\grad\bfu|^2\,\dx\right)\,\mathrm{d}s,\\
		&|I_{12}|\le \|\widetilde{\rho}\|_{L^\infty_t(L^\infty_x)}
		\int_0^t\int_\Om\!C(\e)|\grad\rho|^2+\e|\grad\diver\bfu|^2\,\dx\mathrm{d}s,\\
		&|I_{13}|\le C(\e)\int_0^t\|\grad^2\widetilde{\rho}(s,\cdot)\|^2_{L^3_x}\|\grad\rho\|^2_{L^2_x}\,\mathrm{d}s
		+\e\int_0^t\int_\Om\!|\grad\bfu|^2\,\dx\mathrm{d}s,\\
		&|I_{14}|\le C(\e)\int_0^t\!\|\grad\widetilde{\rho}(s,\cdot)\|_{L^\infty_x}\int_\Om\!|\grad\rho|^2\,\dx\mathrm{d}s
		+\e\int_0^t\int_\Om\!|\grad\bfu|^2\,\dx\mathrm{d}s.
	\end{align*}
 These estimates are now combined with the energy inequality \eqref{iienergy} to deduce
	\begin{equation*}\begin{aligned}
		&\int_\Om\frac{1}{2}\left(|\rho|^2+|\grad\rho|^2+|\bfu|^2\right)(t,x)\,\dx
		+\int_0^t\int_\Om\!\left(\mu|\grad\bfu|^2+(\mu+\lambda)|\grad\diver\bfu|^2\right)\dx\mathrm{d}s\\
		&\lesssim \int_\Om\!\frac{1}{2}\left(|\rho_0|^2+|\grad\rho_0|^2+|\bfu_0|^2\right)\,\dx
		+\int_0^t\! A(s)\int_\Om\!\frac{1}{2}\left(|\rho|^2+|\grad\rho|^2+|\bfu|^2\right)\,\dx\mathrm{d}s\\
		&\quad+\frac{1}{2}\|\widetilde{\rho}\|_{L^\infty_t(L^\infty_x)}\int_0^t\int_\Om\!|\grad\diver\bfu|^2\,\dx\mathrm{d}s
		+ C(\e)\int_0^t\int_\Om\!|\bfU|^2\,\dx\mathrm{d}s.
	\end{aligned}\end{equation*}
where the $\e$-terms were absorbed into the left-hand side of \eqref{iienergy} by choosing $\e$ small enough, and where
$A(\cdot)\in L^1(0,T)$ depends on $\e$, $\|\grad\widetilde{\bfu}\|_{L^\infty_x}$, 
$\|\Pt\widetilde{\bfu}+\widetilde{\bfu}\cdot\grad\widetilde{\bfu}\|_{L^{3/2}_x}^2$, $\|\bff\|_{L^{3/2}_x}^2$,
$\|p'(\widetilde{\rho})\|_{L^\infty_t(L^\infty_x)}$, $\|\widetilde{\rho}\|_{L^\infty_t(L^\infty_x)}$,
$\|\grad\diver\widetilde{\bfu}\|_{L^3_x}^2$, $\|\grad\widetilde{\rho}\|_{L^\infty_x}^2$, and
$\|\grad^2\widetilde{\rho}\|_{L^3_x}^2$.
This concludes the proof.
\end{proof}
Our goal is to eventually apply Gr{\"o}nwall's lemma to \eqref{en1}, but first more estimates are needed on $\grad^2\bfu$. 
	\begin{proposition}\label{prop2}
		Suppose the assumptions of Theorem \ref{mainthm} are satisfied. 
		Then any regular (smooth) solution $(\rho,\bfu)$ of the linearized system \eqref{NSlin} satisfies 
		the following energy inequality for all $t\in (0,T)$:
			\begin{equation*}\begin{aligned}
		&\int_\Om\!\left(\frac{\mu}{2}|\grad\bfu|^2+\frac{\mu+\lambda}{2}|\diver\bfu|^2\right)\!(t,x)\,\dx
			+\int_0^t\int_\Om\!\left(|\Pt\bfu|^2+|\diver\mathbb{S}(\grad\bfu)|^2\right)\dx\mathrm{d}s\\
			&\lesssim \int_\Om\!\left(\frac{\mu}{2}|\grad\bfu_0|^2+\frac{\mu+\lambda}{2}|\diver\bfu_0|^2\right)\dx
				+ \int_0^t\! B_1(s)\|\rho\|^2_{W^{1,2}_x}\mathrm{d}s
			+\int_0^t\! B_2(s)\int_\Om\!|\grad\bfu|^2\,\dx\mathrm{d}s,
			\end{aligned}\end{equation*}
	\end{proposition}
	where the coefficient $B_1\in L^1(0,T)$ depends on $\|\bff\|_{L^3_x}^2$, 
	$\|\Pt\widetilde{\bfu}+\widetilde{\bfu}\cdot\grad\widetilde{\bfu}\|_{L^3_x}^2$, 
	$\|\grad p'(\widetilde{\rho})\|_{L^3_x}^2$, $\|p'(\widetilde{\rho})\|_{L^\infty_x}^2$, $\|p'(\widetilde{\rho})\grad p'(\widetilde{\rho})\|_{L^3_x}^2$, and $B_2\in L^1(0,T)$ depends on $M$, 
	$\|\grad\widetilde{\bfu}\|_{L^\infty_x}^2$, $\|\widetilde{\bfu}\|_{L^\infty_x}^2$.
\begin{proof}
The following estimate is inspired by H. Beir{\~a}o da Veiga \cite{dv83}. 
Begin by rewriting equation \eqref{l2} in the form
	\begin{equation}\label{btrick}
		\widetilde{\rho}\Pt\bfu - \diver\mathbb{S}(\grad\bfu)={\bf g},
	\end{equation}
where
	\begin{equation*}
		{\bf g}= \rho\bff+\bfU-\rho(\Pt\widetilde{\bfu}+\widetilde{\bfu}\cdot\grad\widetilde{\bfu})-\grad(\rho p'(\widetilde{\rho}))
		-\widetilde{\rho}(\bfu\cdot\grad\widetilde{\bfu}-\widetilde{\bfu}\cdot\grad\bfu)
	\end{equation*}
contains the remaining terms. Let $\e>0$ and take the scalar product of \eqref{btrick} with $\Pt\bfu-\e\diver\mathbb{S}(\grad\bfu)$ to get
	\begin{equation}\label{bt0}
		\int_\Om\!(\widetilde{\rho}\Pt\bfu - \diver\mathbb{S}(\grad\bfu))\cdot(\Pt\bfu-\e\diver\mathbb{S}(\grad\bfu))\,\dx
		=\int_\Om\!{\bf g}\cdot(\Pt\bfu-\e\diver\mathbb{S}(\grad\bfu))\,\dx.
	\end{equation}
Integrating \eqref{bt0} in time, integrating by parts in space, and using H{\"o}lder's inequality we get
	\begin{equation}\label{bt1}\begin{aligned}
		&\int_\Om\!\left(\frac{\mu}{2}|\grad\bfu|^2+\frac{\mu+\lambda}{2}|\diver\bfu|^2\right)(t,x)\,\dx
		+\int_0^t\int_\Om\!\left(m|\Pt\bfu|^2+\e|\diver\mathbb{S}(\grad\bfu)|^2\right)\dx\mathrm{d}s\\
		&\le \int_\Om\!\left(\frac{\mu}{2}|\grad\bfu_0|^2+\frac{\mu+\lambda}{2}|\diver\bfu_0|^2\right)\,\dx\\
		&\quad+\int_0^t\int_\Om\!\left(|{\bf g}||\Pt\bfu|+\e|{\bf g}||\diver\mathbb{S}(\grad\bfu)|+\e|\widetilde{\rho}||\Pt\bfu|
		|\diver\mathbb{S}(\grad\bfu)|\right)\dx\mathrm{d}s\\
		&\le \int_0^t\!\left(\|{\bf g}\|_{L^2_x}\|\Pt\bfu\|_{L^2_x}+\e\|{\bf g}\|_{L^2_x}\|\diver\mathbb{S}(\grad\bfu)\|_{L^2_x}
			+\e M \|\Pt\bfu\|_{L^2_x}\|\diver\mathbb{S}(\grad\bfu)\|_{L^2_x}\right)\,\mathrm{d}s,
	\end{aligned}\end{equation}
where we also used that $0<m\le\widetilde{\rho}\le M<+\infty$. Next we estimate the right-hand side of \eqref{bt1} using Young's inequality:
	\begin{equation}\label{bt2}\begin{aligned}
		\|{\bf g}\|_{L^2_x}\|\Pt\bfu\|_{L^2_x} &\le \frac{m}{8}\|\Pt\bfu\|^2_{L^2_x}+\frac{2}{m}\|{\bf g}\|^2_{L^2_x},\\
		\e\|{\bf g}\|_{L^2_x}\|\diver\mathbb{S}(\grad\bfu)\|_{L^2_x} 
		&\le \frac{\e}{4}\|\diver\mathbb{S}(\grad\bfu)\|^2_{L^2_x}+\e\|{\bf g}\|^2_{L^2_x},\\
		\e M \|\Pt\bfu\|_{L^2_x}\|\diver\mathbb{S}(\grad\bfu)\|_{L^2_x}
		&\le 4\e M^2 \|\Pt\bfu\|^2_{L^2_x}+\frac{\e}{4}\|\diver\mathbb{S}(\grad\bfu)\|^2_{L^2_x}.
	\end{aligned}\end{equation}
Since $\e>0$ is arbitrary, we choose $\e< 7m/(32M^2)$ in order to absorb the estimates \eqref{bt2} into the left-hand side of \eqref{bt1}. Having chosen this $\e$, inserting the estimates \eqref{bt2} into \eqref{bt1} we get
		\begin{equation}\label{bt3}\begin{aligned}
			&\int_\Om\!\left(\frac{\mu}{2}|\grad\bfu|^2+\frac{\mu+\lambda}{2}|\diver\bfu|^2\right)(t,x)\,\dx
			+\int_0^t\int_\Om\!\left(|\Pt\bfu|^2+|\diver\mathbb{S}(\grad\bfu)|^2\right)\dx\mathrm{d}s\\
			&\lesssim \int_\Om\!\left(\frac{\mu}{2}|\grad\bfu_0|^2+\frac{\mu+\lambda}{2}|\diver\bfu_0|^2\right)\,\dx
			+ \int_0^t\int_\Om\!|{\bf g}|^2\,\dx\mathrm{d}s.
		\end{aligned}\end{equation}
Finally let us estimate the ${\bf g}$ term. By virtue of Minkowski's inequality,
	\begin{equation*}\begin{aligned}
		\|{\bf g}\|_{L^2_t(L^2_x)}^2 &\lesssim \|\rho\bff\|^2_{L^2_t(L^2_x)}+\|\bfU\|^2_{L^2_t(L^2_x)}
		+\|\rho(\Pt\widetilde{\bfu}+\widetilde{\bfu}\cdot\grad\widetilde{\bfu})\|^2_{L^2_t(L^2_x)}\\
		&\quad+\|\grad(\rho p'(\widetilde{\rho}))\|^2_{L^2_t(L^2_x)}
		+\|\widetilde{\rho}(\bfu\cdot\grad\widetilde{\bfu}-\widetilde{\bfu}\cdot\grad\bfu)\|^2_{L^2_t(L^2_x)}\\
		&=:\sum_{i=1}^5 J_i.
	\end{aligned}\end{equation*}
Employing the standard H{\"o}lder's inequality, Young's inequality, Poincar{\'e} inequality, and Sobolev embedding 
$W^{1,2}_x\subset L^6_x$, we estimate each of the $J_i$ as follows: 
	\begin{equation}\label{estj}\begin{aligned}
		&J_1\le \int_0^t\!\|\bff\|^2_{L^3_x}\|\rho\|^2_{L^6_x}\,\mathrm{d}s
			\le C\int_0^t\!\|\bff\|^2_{L^3_x}\|\rho\|^2_{W^{1,2}_x}\,\mathrm{d}s,\\
		&J_2 = \|\bfU\|^2_{L^2_t(L^2_x)}<+\infty,\\
		&J_3\le \int_0^t\!\|\Pt\widetilde{\bfu}+\widetilde{\bfu}\cdot\grad\widetilde{\bfu}\|^2_{L^3_x}
			\|\rho\|^2_{L^6_x}\,\mathrm{d}s
			\le C\int_0^t\! \|\Pt\widetilde{\bfu}+\widetilde{\bfu}\cdot\grad\widetilde{\bfu}\|^2_{L^3_x}
			\|\rho\|^2_{W^{1,2}_x}\,\mathrm{d}s,\\
		&J_4 = \int_0^t\int_\Om\!\left(|\rho|^2|\grad p'(\widetilde{\rho})|^2
			+2\rho\grad\rho\cdot p'(\widetilde{\rho})\grad p'(\widetilde{\rho})
			+|p'(\widetilde{\rho})|^2|\grad\rho|^2\right)\,\dx\mathrm{d}s\\
		&\quad\le C\int_0^t\!\|\grad p'(\widetilde{\rho})\|^2_{L^3_x}\|\rho\|^2_{W^{1,2}_x}\,\mathrm{d}s
			+C\int_0^t\!\|\rho\|^2_{W^{1,2}_x}\,\mathrm{d}s\\
		&\qquad+\int_0^t\!\|p'(\widetilde{\rho})\grad p'(\widetilde{\rho})\|^2_{L^3_x}
			\|\grad\rho\|^2_{L^2_x}\,\mathrm{d}s+\int_0^t\!
			\|p'(\widetilde{\rho})\|^2_{L^\infty_x}\int_\Om\!|\grad\rho|^2\,\dx\,\mathrm{d}s,\\
		&J_5\le M\int_0^t\!\|\grad\widetilde{\bfu}\|^2_{L^\infty_x}\int_\Om\!|\bfu|^2\,\dx\mathrm{d}s
			+M\int_0^t\!\|\widetilde{\bfu}\|^2_{L^\infty_x}\int_\Om\!|\grad\bfu|^2\,\dx\mathrm{d}s.
	\end{aligned}\end{equation}
Inserting the estimates \eqref{estj} into inequality \eqref{bt3} and rearranging, we get that
	\begin{equation*}\begin{aligned}
		&\int_\Om\!\left(\frac{\mu}{2}|\grad\bfu|^2+\frac{\mu+\lambda}{2}|\diver\bfu|^2\right)(t,x)\,\dx
			+\int_0^t\int_\Om\!\left(|\Pt\bfu|^2+|\diver\mathbb{S}(\grad\bfu)|^2\right)\dx\mathrm{d}s\\
			&\lesssim \int_\Om\!\left(\frac{\mu}{2}|\grad\bfu_0|^2+\frac{\mu+\lambda}{2}|\diver\bfu_0|^2\right)\,\dx
			+ \|\bfU\|^2_{L^2_t(L^2_x)}\\
			&\quad+ \int_0^t\! B_1(s)\|\rho\|^2_{W^{1,2}_x}\mathrm{d}s
			+\int_0^t\! B_2(s)\int_\Om\!|\grad\bfu|^2\,\dx\mathrm{d}s,
	\end{aligned}\end{equation*}
where the coefficient $B_1\in L^1(0,T)$ depends on $\|\bff\|_{L^3_x}^2$, 
$\|\Pt\widetilde{\bfu}+\widetilde{\bfu}\cdot\grad\widetilde{\bfu}\|_{L^3_x}^2$, 
$\|\grad p'(\widetilde{\rho})\|_{L^3_x}^2$, $\|p'(\widetilde{\rho})\|_{L^\infty_x}^2$, $\|p'(\widetilde{\rho})\grad p'(\widetilde{\rho})\|_{L^3_x}^2$, and $B_2\in L^1(0,T)$ depends on $M$,
$\|\grad\widetilde{\bfu}\|_{L^\infty_x}^2$, $\|\widetilde{\bfu}\|_{L^\infty_x}^2$.
The proof is complete.
\end{proof}
The next result provides the existence and regularity for the Lam{\'e} system, used to obtain the $W^{2,2}$ velocity estimates. It may be found in \cite{ns04}, Lemma 4.32.  
	\begin{proposition}\label{lame}
		Let $\Omega\subset\mathbb{R}^3$ be a bounded domain of class $C^2$, suppose 
		\[\mu>0,\quad 4\mu+3\lambda>0,\]
		and let ${\bf{F}}\in L^2(\Om;\mathbb{R}^3)$ be given. 
		Then there exists a unique strong solution $\bfu\in W^{2,2}(\Om;\mathbb{R}^3)\cap W^{1,2}_0(\Om;\mathbb{R}^3)$ satisfying 
			\begin{equation*}\begin{aligned}
				-\mu\Delta\bfu -(\mu+\lambda)\grad\diver\bfu &= {\bf{F}}\,\,\mathrm{in}\,\,\Omega,\\
				\bfu&=0\,\,\mathrm{on}\,\,\partial\Omega.
			\end{aligned}\end{equation*}
		Furthermore,
			\[\|\bfu\|_{W^{2,2}(\Om;\mathbb{R}^3)}\le C(\Om)\|{\bf{F}}\|_{L^2(\Om;\mathbb{R}^3)}.\]
	\end{proposition}
Proposition \eqref{lame} also has an $L^p$ version for $1<p<\infty$ (cf. \cite{ns04} Lemma 4.32). We conclude this section with the following lemma concerning the \emph{a priori} estimates.
	\begin{lemma}\label{apriori}
		Define 
			\[\mathcal{E}(\rho,\bfu):= \frac{1}{2}\left(\rho^2+|\grad\rho|^2+|\bfu|^2+\mu|\grad\bfu|^2+(\mu+\lambda)|\diver\bfu|^2\right).\]
		Suppose the assumptions of Theorem \ref{mainthm} are satisfied. Then any regular (smooth) solution $(\rho,\bfu)$ 
		of the linearized system \eqref{NSlin} satisfies the following estimate for any $t\in (0,T)$:
			\begin{equation*}\begin{aligned}
				&\int_\Om\!\mathcal{E}(\rho,\bfu)(t)\,\dx+\int_0^t\int_\Om\!\left(|\Pt\bfu|^2+|\grad^2\bfu|^2\right)\dx\mathrm{d}s\\
				&\lesssim \int_\Om\!
				\left(\mathcal{E}(\rho_0,\bfu_0)\,\dx + \|\bfU\|_{L^2_t(L^2_x)}\right)
				\cdot\exp\left(\int_0^t\! A(s)+B_1(s)+B_2(s)\,\mathrm{d}s\right),
			\end{aligned}\end{equation*}
		where $A, B_1, B_2\in L^1(0,T)$, depending on the data $\widetilde{\rho}, \widetilde{\bfu},$ 
		are defined in Propositions \ref{prop1} and \ref{prop2}. 
	\end{lemma}
	\begin{proof}
		Using Proposition \ref{lame}, we retrieve the estimate 
			\begin{equation*}
			\|\bfu\|_{W^{2,2}_x}\le C(\Om)\|\diver\mathbb{S}(\grad\bfu)\|_{L^2_x}.
			\end{equation*}
		Combining this estimate with the results of Propositions \ref{prop1} and \ref{prop2}, we get that
			\begin{equation*}\begin{aligned}
				&\int_\Om\!\mathcal{E}(\rho,\bfu)(t)\,\dx + \int_0^t\int_\Om\!|\Pt\bfu|^2+|\grad^2\bfu|^2\,\dx\mathrm{d}s\\
				&\lesssim \int_\Om\!\mathcal{E}(\rho_0,\bfu_0)\,\dx +\|\bfU\|_{L^2_t(L^2_x)}^2 +\int_0^t\!(A(s)+B_1(s)+B_2(s))
				\int_\Om\!\mathcal{E}(\rho,\bfu)\,\dx\mathrm{d}s.
			\end{aligned}\end{equation*}
		An application of Gr{\"o}nwall's lemma concludes the proof.
	\end{proof}
\subsection{Proof of Theorem \ref{mainthm}}
We proceed using a standard fixed point method in the spirit of \cite{fnp01}. 

First note that since the velocity $\widetilde{\bfu}$ is smooth up to the boundary, the density equation \eqref{l1}, or
	\begin{equation*}\begin{aligned}
		&\Pt\rho+\widetilde{\bfu}\cdot\grad\rho = -\rho\diver{\widetilde{\bfu}}-\diver(\widetilde{\rho}\bfu),\\
		&\rho(0,\cdot)=\rho_0,
	\end{aligned}\end{equation*}
has a unique solution $\rho=\rho[\bfu]$ in $C([0,T];H^1(\Om))$ by the method of characteristics (cf. \cite{sol76,valli83}), and by 
linearity this mapping is continuous.
Denote by $X_n=\mathrm{span}\{\pi_j\}_{j=1}^n$ a finite dimensional Hilbert space (with $L^2(\Om)$ inner product), 
where $\{\pi_j\}_{j=1}^\infty\subset C^\infty_0(\Om;\mathbb{R}^3)$ is dense in $C^2_0(\overline{\Om};\mathbb{R}^3)$.

We seek a fixed point $\bfu_n\in C([0,T];X_n)$ of the problem
	\begin{equation}\label{ap1}
		\bfu_n(t)=\bfu_0 + \int_0^t\!\mathcal{M}(\bfu_n(s),\rho[\bfu_n](s))\,\mathrm{d}s =: \mathcal{T}[\bfu_n],
	\end{equation}
where 
	\begin{equation*}\begin{aligned}
	\mathcal{M}(\bfu_n,\rho[\bfu_n]) 
	:= &-\bfu_n\cdot\grad\widetilde{\bfu}-\widetilde{\bfu}\cdot\grad\bfu_n-(\widetilde{\rho})^{-1}\grad(\rho[\bfu_n] p'(\widetilde{\rho}))
		-(\widetilde{\rho})^{-1}\diver\mathbb{S}(\grad\bfu_n)\\
		&+(\widetilde{\rho})^{-1}[-\rho[\bfu_n](\Pt\widetilde{\bfu}+\widetilde{\bfu}\cdot\grad\widetilde{\bfu}_n)+\rho[\bfu_n]\bff+\bfU],
	\end{aligned}\end{equation*}
contains the remainder of the momentum equation \eqref{l2}.
\begin{remark}
Strictly speaking we should first project $\bfu_0$ and $\mathcal{M}$ to $X_n$ but we ignore this point here. 
See \cite{ns04}, Section 7.7.2 for more details.
\end{remark}
Next define the ball 
	\[\mathcal{B}=\left\{\bfv\in C([0,T];X_n)\,\big|\, \sup_{t\in[0,T]}\|\bfv(t)-\bfu_0\|_{X_n}\le 1\right\},\]
which is closed, convex, and bounded. By choosing $T=T'$ small enough, the mapping $\mathcal{T}$ maps $\mathcal{B}$ into itself. Furthermore,
it may be shown that the family $\{\mathcal{T}[\bfu_n]\}_{n\ge 1}$ is equicontinuous in $C([0,T];X_n)$, hence the Arzel{\'a}-Ascoli theorem
applies and $\mathcal{T}$ is a compact operator. We deduce by Schauder's theorem that there exists a fixed point $\bfu_n$ of \eqref{ap1}.

The estimates in Lemma \ref{apriori}, being uniform up to time $T$, allow to extend the time interval for existence up to this time. Similarly,
uniformity in $n$ allows us to extract subsequences, still denoted $\rho_{n}$, $\bfu_{n}$, such that
	\begin{equation}\label{a55}\begin{aligned}
		&\rho_n\weak^* \rho\quad\,\mathrm{in}\,\, L^\infty_t(H^1_x),\\
		&\bfu_n\weak^* \bfu\quad\mathrm{in}\,\, L^\infty_t(H^1_{0,x}),\\
		&\bfu_n\weak \bfu\quad\,\,\mathrm{in}\,\, L^2_t(H^2_x).
	\end{aligned}\end{equation}
Due to this regularity and linearity of the linearized Navier-Stokes system, we may pass to the limit to conclude $(\rho,\bfu)$ is a strong solution. This concludes the proof.
\end{appendix}

\section*{Acknowledgments}
S. Doboszczak acknowledges support by the Postgraduate Research Participation Program at U.S. Air Force Institute of Technology (USAFIT), administered by the Oak Ridge Institute for Science and Education through an
interagency agreement between the U.S. Department of Energy and USAFIT.
M.T. Mohan would like to thank the Air Force Office of Scientific
Research and National Research Council for support through a Research
Associateship Award, and the USAFIT
for providing stimulating scientific environment and resources. S.S. Sritharan's work has been funded by U. S.
Army Research Office, Probability and Statistics program.


\end{document}